\documentclass[oneside,a4paper,reqno, english, 10 pt]{article}
\usepackage{amssymb,amsmath,amsthm} %,hyperref}
\usepackage{makeidx}
\usepackage{pifont}
\usepackage{bbding}
\usepackage{color}
\definecolor{OliveGreen}{rgb}{0.41, 0.16, 0.38}
\usepackage{textcomp}
\usepackage{ esint }
\usepackage{graphicx}
\usepackage{pgf,tikz}
\usepackage[english]{babel}
\usepackage{soul}

   \newcommand{\beq}{\begin{equation}}
   \newcommand{\eeq}{\end{equation}}
   \newcommand{\beqs}{\arraycolsep1.5pt\begin{eqnarray}}
   \newcommand{\eeqs}{\end{eqnarray}\arraycolsep5pt}
   \newcommand{\beqsn}{\arraycolsep1.5pt\begin{eqnarray*}}
   \newcommand{\eeqsn}{\end{eqnarray*}\arraycolsep5pt}

  \xdefinecolor{azzurro}{rgb}{0.2,0.9,0.9}%azzurro
\xdefinecolor{violetto}{rgb}{0.5,0.5,0.9}%violetto
\xdefinecolor{marrone}{rgb}{0.8,0.4,0.2}%marroncino

\newtheorem{pft}{}                              % proof

\newcommand{\bpf}{\begin{pft}\begin{proof}\mbox{\bf Proof}.\ \rm}
\newcommand{\bpfof}[1]{\begin{pft}\begin{proof}{\bf Proof of #1}\
\rm}
\newcommand{\bspf}{\begin{pft}\begin{proof}{\bf Sketch of proof}.\
\rm}

\newcommand{\epf}{\nopagebreak\end{proof}\end{pft}}

\newtheorem{thm}{Theorem}[section]
\newtheorem{rem}[thm]{Remark}
\newtheorem{cor}[thm]{Corollary}
\newtheorem{prop}[thm]{Proposition}
\newtheorem{lemma}[thm]{Lemma}
\newtheorem{defn}[thm]{Definition}

\numberwithin{equation}{section}

\def \trait (#1) (#2) (#3){\vrule width #1pt height #2pt depth #3pt}
\def \qed{\hfill
        \trait (0.1) (6) (0)
        \trait (6) (0.1) (0)
        \kern-6pt
        \trait (6) (6) (-5.9)
        \trait (0.1) (6) (0)
\medskip}
\def\square{\trait (0.1) (6) (0)
        \trait (6) (0.1) (0)
        \kern-6pt
        \trait (6) (6) (-5.9)
        \trait (0.1) (6) (0)}
\def\Bbb#1{{\fam\msbfam\relax#1}}

\def\R{\Bbb R}
\def\N{{\Bbb N}}

\newcommand{\ds}{\rightarrow}

\newcommand{\om}{\Omega}
\newcommand{\Om}{\Omega}
\newcommand{\f}{\varphi}

\newcommand{\eps}{\varepsilon}

\def\ssubset{\subset \subset}

\def\B{{\cal B}}

\def\L{{\cal L}}

\def\infess{\mathop{\rm ess\: inf }}
\def\supess{\mathop{\rm ess\: sup }}

\def\supp{{\rm supp}\,}

\newcommand{\wto}{\rightharpoonup}

\newcommand{\loc}{{\rm loc}}

\def\supess{\mathop{\rm ess\: sup }}

\font\tenmsb=msbm10 \font\sevenmsb=msbm7 \font\fivemsb=msbm5
\newfam\msbfam
\textfont\msbfam=\tenmsb \scriptfont\msbfam=\sevenmsb
\scriptscriptfont\msbfam=\fivemsb
\def\Bbb#1{{\fam\msbfam\relax#1}}

\def\to{\rightarrow}

\newenvironment{michelarev}{\color{blue}}{\color{black}}
\newcommand{\bmicr}{\begin{michelarev}}
\newcommand{\emicr}{\end{michelarev}}

\title{\bf Asymptotic analysis of thin structures with \\
point dependent energy growth}
\author{{\sc Michela Eleuteri}\\
Universit\`a degli Studi di Modena e Reggio Emilia\\
Dipartimento di Scienze Fisiche Informatiche e Matematiche \\
 via Campi 213/b, 41125 Modena (Italy)\\
{\tt michela.eleuteri@unimore.it}
\\[4mm] 
 {\sc Francesca Prinari}\\
 Universit\`a degli Studi di Pisa  \\
 Dipartimento di Scienze Agrarie, Alimentari e Agro-ambientali
\\
 via del Borghetto, Pisa (Italy)
\\
 {\tt francesca.prinari@unipi.it}
\\[4mm] 
{\sc Elvira Zappale}\\
 Sapienza -  Universit\`a di Roma  \\
 Dipartimento di Scienze di Base e Applicate per l'Ingegneria\\
 via Antonio Scarpa, 16, 00161, Roma (Italy)
\\
 {\tt  elvira.zappale@uniroma1.it}
}

\begin{document}
\maketitle
%\begin{center}
%\date{\fbox{\today}}
%\end{center}
\begin{abstract}
$3d-2d$ dimensional reduction for hyperelastic thin films modeled through energies with point dependent growth, assuming that the sample is clamped on the lateral boundary, is performed in the framework of $\Gamma$-convergence. Integral representation results, with a more regular lagrangian related to the original energy density, are provided for the lower dimensional limiting energy, in different contexts.
\end{abstract}


{\small
\tableofcontents
}
%\begin{itemize}
%	\item  {\color{OliveGreen} Elvira}
%	\item { \color{magenta} Michela: sistemare i trattini e i $k$}
%	\item  {\color{blue} Francesca}
%\end{itemize}

\section{Introduction}

%\label{Intro}

%\color{red} Mettere spazio funzionale giusto all'inizio e mettere eventualmente non unicita' del limite nel teorema  di convergenza dei minimi.
%\color{black}
%Our first aim consists in the study of sufficient conditions which guarantee the existence of solutions to several families of equilibrium problem. 

In solid mechanics, the equilibrium state of a body can be described by an energy minimization problem, over a suitable class of fields, of an integral functional of the type %\begin{equation}
%\label{miofunzionale} 
\[
\int_{{\Omega}}
f(x, \nabla v(x))dx,
\]
%\end{equation} 
where $\Omega \subset \mathbb R^N$ is a bounded open set, $v$ is the deformation or the displacement, defined in a suitable function space $X(\Omega;\mathbb R^d)$, e.g. $ W^{1,1}_{\loc}(\Omega;\mathbb R^d)$,  and $f:\Omega \times \mathbb{R}^{d \times N} \rightarrow \mathbb{R}$ is the hyperelastic energy density, assumed as a Carath\'{e}odory function satisfying suitable growth conditions. 

In many applications, for instance to deal with multiphase materials, finer and finer heterogeneities,  shape optimization, thin structures, phase transitions, etc. other parameters may  come into play and the equilibrium configurations arise as solutions of the following $\varepsilon$- parameterized problems  \color{black}
\begin{align}\label{genminpb1}
	\min_{v \in X(\varepsilon)(\Omega_\varepsilon;\mathbb R^d)} {\mathcal{F}}_\varepsilon(v):= 	\min_{v \in X(\varepsilon)(\Omega_\varepsilon;\mathbb R^d)} \frac 1 \varepsilon 
 \int_{\Omega_\varepsilon} f(\varepsilon)(x,\nabla v(x))dx
\end{align}
where $\{\Omega_\varepsilon\}$ may represent a $\varepsilon$- dependent family of bounded open subsets in $\mathbb R^N$ and  $X(\varepsilon)(\Omega_\varepsilon;\mathbb R^d)$ is a suitable functional space,  possible encoding the anisotropy of the model, expressed through a point-dependent integrability condition on the deformation gradients.
Hence one is interested in detecting, as $\varepsilon \to 0^+$, the asymptotic behaviour of such configurations. 

\noindent

Two main types of questions are addressed in this paper, one dealing with the detection of sufficient conditions on $X(\varepsilon)$, that, under special structure conditions on $f(\varepsilon),$ ensure existence of solutions of the  problem  \eqref{genminpb1}, when the domain $\Omega_\varepsilon$ is a thin domain of the form 
\begin{align}\label{Omegaepsilon}
\Omega_\varepsilon:=\omega \times (-\varepsilon/2, \varepsilon/2),
\end{align}
with $\omega \subset \mathbb R^2$ a bounded, connected,  open set with Lipschitz  boundary,
the other one concerning the asymptotic behaviour of problems \eqref{genminpb1}, under more general assumptions on $f(\varepsilon)$. In the latter case, by means of a  suitable variational convergence, we derive a limiting problem defined on $\omega$ with a structure analogous to the original one, which  admits solutions.  
\color{black}
This approach rigorously  justifies the fact that, in the applications, when dealing with
very thin structures, it is convenient to work with a lower dimensional model. We refer to the pioneering papers \cite{ABP} and \cite{LDR} where this approach has been formally obtained, in the framework of $\Gamma$-convergence, in the hyperelastic setting for strings and membranes, in the classical functional setting of Sobolev spaces.  

The main novelty of this manuscript is that the energy density $f(\varepsilon)$ depends on the parameter $\varepsilon$ also through its growth which involves a so called \textit{variable growth exponent} $p_{\eps}=p_{\eps}(x)$  or better, a family of variable growth exponents converging to a fixed one in some suitable topology. This leads to non-trivial mathematical questions, because it appears to be challenging the idea to control the dependence of the parameter $\varepsilon$ both in the set and in the energy density. 

In this context it is convenient to rephrase the problem \eqref{genminpb1} on a fixed domain 
\begin{equation}\label{Omega1}
\Omega := \Omega_1=\omega \times (-1/2,1/2)
\end{equation} through a re-scaling in the transverse direction $x_3$, thus obtaining the equivalent minimization problem 
\begin{align}\label{genminpbresc}
	\min_{u \in X_\varepsilon(\Omega;\mathbb R^3)}\mathcal{F}_\varepsilon(u):= 	\min_{u \in X_\varepsilon(\Omega;\mathbb R^3)}\int_{\Omega} f_\varepsilon\left(x,\nabla u(x)\right)dx,
\end{align}
where $f_\varepsilon,$ $X_\varepsilon(\Omega;\mathbb R^3)$, and $u$ have been obtained by $f(\varepsilon)$,  $X(\varepsilon)(\Omega_\varepsilon;\mathbb R^d)$ with $d = 3$ and $v$, respectively, via the change of variables which maps $x_3$ from $(-\varepsilon/2,\varepsilon/2)$ into $(-1/2,1/2)$.
In the first case,  we  consider  $p_\varepsilon :\Omega \to (1,+\infty)$, $X_{\varepsilon}(\omega; \R^3)=W^{1,p_{\varepsilon}(\cdot)}(\Omega; \mathbb{R}^3)$ and $f_\varepsilon: \Omega \times \mathbb R^{3 \times 3}\to \mathbb R$  of the following type
%\begin{equation}
%\label{geps}
\[
f_\varepsilon(x,\xi):=W^{p_\varepsilon(x)}(x,\xi),\quad \hbox{or } \quad 
f_\varepsilon (x,\xi):=a_\varepsilon(x)W^{p_\varepsilon(x)}(\xi),
\]
%\end{equation}
with  suitable functions 
  $W$ and $a_\varepsilon$   (which are made precise in the sequel). In particular we provide lower semicontinuity 
 with respect to the (strong $L^\infty \times L^1$)  joint convergence of the sequences $\{p_{\eps}\} \subseteq L^\infty(\Omega;(1,+\infty))$ and  $\{u_{\eps}\}\subseteq W^{1,1}(\Omega; \mathbb{R}^d)$  (see Theorems \ref{semicon} and \ref{semicon2}).

In the  second case, we  deal with the very  dimensional reduction problem  where  the function $f_\varepsilon$ in \eqref{genminpbresc}, representing the hyperelastic energy density of a $3d$ thin structure,
is of the form
\begin{align}\label{fgen}
f_\varepsilon(x, \xi):= f(x_\alpha, \varepsilon x_3, p_\varepsilon (x), \xi_\alpha, \tfrac{1}{\varepsilon} \xi_3),
\end{align} for every 
%\begin{align}
%\label{xalpha}
\[
x\equiv (x_\alpha, x_3)\in \Omega,\,
\xi \equiv (\xi_\alpha, \xi_3)\in \mathbb R^{3 \times 3},  \alpha =1,2. 
%\end{align}
\]
Here   $p_\varepsilon$ plays the role of a variable exponent defined in $\Omega$, i.e.
\begin{equation}\label{pepsi}
p_\varepsilon(x)\equiv p_\varepsilon(x_\alpha, x_3):=p(x_\alpha, \varepsilon x_3)
\end{equation} for ${\mathcal L^3}$- a.e. $x \in \Omega$, and the re-scaled variable $(\xi_\alpha, \frac{1}{\varepsilon}\xi_3) \in \mathbb R^{3\times 3}$\, appears to take into account the scaling of the gradient $\nabla v=(\nabla_{\alpha}v,\nabla_{3}v)$  after the change of variable which maps it in the scaled deformation gradient of $u$ in \eqref{genminpbresc}, cf. Section \ref{secdimred} for precise definition.

Dimensional reduction problems of the type \eqref{genminpb1} have attracted much attention in the past decades due to the many applications in engineering, materials science, conductors, micromagnetics, chemistry and biology. Indeed a wide literature has been developed to rigorously deduce a simplified formulation, in a lower dimensional setting via a variational approach.  Among a wide literature, we refer to \cite{AL, GH} for micromagnetic and ferromgnetic materials, to \cite{ARS2}, \cite{BrF} for brittle materials, to \cite{CMMO} in the case of materials which allow for fracture and plastic behaviour, to \cite{KSZ, DKPS} in the case of magnetoelasticity, to \cite{FMP, GZNODEA, GGP, GZ2, H, KP} for the case of hyperelastic multistructures, to \cite{FPRZ}, to model delamination, to \cite{FZ, V} to detect bending effect, in particular in shells as in \cite{HV}, or to  \cite{BF, BFF, BD} to describe thin structures made by heterogeneous materials with fine distribution, also  in the discrete setting. 
Besides the above far from exhaustive bibliography, it is worth to point out that particular minimum problems like the one in \eqref{genminpbresc} appear also in the modeling of conductors to detect the dielectric breakdown, see e.g. \cite{BPZ}, or in the context of damaging and optimal design (cf. in \cite{BFF}, \cite{CZCRAS},  and \cite{KZ}). 
It is worth to point out that, besides it has been analyzed the case of multicomponent materials which exhibit a very different behaviour from one point to another, and imposing a perimeter penalization among the interfaces of the components, see \cite{BZ1,BZ2}, at the best of our knowledge, it has been always considered the case where the energetic pointwise dependent behaviour has a rough change from one point to another of the sample. The case where the position dependent energetic behaviour is suitably continuous has not yet been considered. This is indeed one of the topic of this paper, in the pure hyperelastic setting, neglecting damage or shape optimization. Indeed, despite the last mentioned results, the regularity of the energetic point dependent behaviour allows us to obtain an explicit representation of the limiting problems in any dimensional reduction setting.
\vspace{2mm}
\color{black}
%The energy variable exponent is nowadays a prominent topic in the Calculus of Variations and Nonlinear Analysis \textcolor{red}{cf.....INSERIRE.}  The  basic model functional  is
%\begin{equation}\label{modello9}
%\int_{\Omega}|\nabla u|^{p(x)}\ dx
%\end{equation}
% $p: \Omega \rightarrow (1, + \infty)$ is a (Lebesgue) measurable function called \textit{(bounded) variable exponent}; more in general one can assume that the functionals as in \eqref{miofunzionale} satisfy the so called {\it $p(x)-$growth condition
%\begin{equation}\label{p2}
%\[
%L^{-1}|\xi|^{p(x)}\leq f(x,\xi) \leq L(1+|\xi|^{p(x)})
%\]
%\end{equation}
%for some constant $L \ge 1$.}

%\vspace{2mm}

%\vspace{2mm}

We aim now to present the main results of the paper. To this purpose, referring to \cite{DHHR11} for a comprehensive treatment of the topic,
we recall that a (Lebesgue) measurable function $p: \Omega \rightarrow [1, + \infty]$,  playing the role of an exponent, is called a {\it variable exponent} and we denote by $\mathcal P(\Omega)$ the class of variable exponents on $\Omega$. Moreover, $\mathcal{P}_b(\Omega)$ and $\mathcal{P}^{log}(\Omega)$ are the subsets of $\mathcal P(\Omega)$ whose elements satisfy the {\it  boundedness condition}
	\begin{align}\label{p-p+}
		1<{p^-}:=\infess_{\Omega}p(x) \leq \supess_{\Omega}p(x)=:{p^+} <+\infty,
	\end{align}
and the so called
{\it  $log$-H\"older continuity property}:
	\begin{align}\label{logp}
		\exists \gamma > 0 : |p(x)-p(y)| \leq \frac{\gamma}{|\log |x-y||} \hbox{ for every }x, y \in \Omega, 0 < |x-y| <\frac{1}{2},
	\end{align}
 respectively. Finally we set $\mathcal{P}^{log}_b(\Omega) :=  \mathcal{P}^{log}(\Omega) \cap \mathcal{P}_b(\Omega)$ (according to Definitions \ref{p} and \ref{pbis}).
Condition \eqref{logp} has been first introduced by  V.V. Zhikov in \cite{Z97}; the failure of this condition is a possible cause of discontinuity of minimizers, while, by assuming it, it is possible to prove higher integrability of minimizers, which is the first step towards further regularity. From that moment onwards, a large number of papers devoted to the mathematical analysis of energy functionals involving variable exponents appeared in several and different directions, motivated by the fact that such types of energies describe models (also non variational) coming from Mathematical Physics that are built using a variable growth exponent. For instance, we refer to the recent contribution \cite{ACFS}, where this theory is employed in the study of elastic problems, to \cite{DCLV} and \cite{SSS}, where the variable exponent growth condition is considered in the framework of free discontinuity problems, i.e. for materials allowing for fractures and damage, 
%On the other hand, functionals like the one in
%\eqref{modello9} have been studied also from %a
%theoretical point of view, since
% they motivate the introduction of
%certain related function spaces with interesting features, like the Lebesgue and Sobolev spaces with variable exponents (see \cite{O31}, \cite{S79}, \cite{KR91}, \cite{ER1}, \cite{ER2}) or generalized Orlicz spaces  \cite{HH19}.
%\vspace{2mm}
%On the other hand, 
to \cite{HHLT}, where a variable exponent model for image restoration has been studied, from the existence point of view and the $\Gamma-$convergence one, in the case that the exponent attains the critical value one, to \cite{EP} where power-law approximation of supremal functionals has been studied once more by means of the $\Gamma-$convergence tool.

\vspace{2mm}

Thus, we are in position to state our main results for the first class of problems. To this end we recall that 
a Borel measurable function $g:\mathbb R^{d \times N} \to \mathbb R$ is said quasiconvex if 
%\begin{equation}\label{qcxdef}
\[
g(\xi)\leq\int_Q g(\xi +\nabla \varphi(x))\,dx 
\]
%\end{equation}
holds, for every $\xi\in\mathbb{R}^{d\times N}$ and for every  $\varphi \in W^{1,\infty}_0(Q;\mathbb R^d)$, where $Q:= \left(-1/2, 1/2\right)^N$ and $W^{1,\infty}_0(Q;\mathbb R^d)=W^{1,\infty}(Q;\mathbb R^d)\cap W^{1,1}_0(Q;\mathbb R^d)$ (we recall that a finite valued quasiconvex function is locally Lipschitz, hence continuous and locally bounded; thus the class of test functions  can be replaced by $C^\infty_c(Q;\mathbb R^d)$ or $C^1_0(Q;\mathbb R^d)$,  as in \cite{MM}).
 \vspace{2mm} In Section \ref{SCIres} we show the following lower semicontinuity. We stress that 
the $L^\infty$ convergence on the variable exponents  is  a natural assumption in this  context. Indeed, such uniformly converging sequences as in Theorems \ref{semicon} and \ref{semicon2} can arise by means of a change of variables as in \eqref{pepsi} from a single exponent $p$ satisfying a uniform log-H\"older continuity assumption.
%in a as happens in the dimensional reduction.

 \begin{thm}\label{semicon} Let $\Omega\subseteq \R^N$ be an  open set. Let $W:\Omega\times \R^{d\times N}\to [0,+\infty)$ be a function such that 
 \begin{itemize}
 \item [-] $W(\cdot,\xi)$ is $\mathcal L^N(\Omega)$-measurable for every $\xi \in \mathbb R^{d \times N}$; 
\item [-]  $W(x, \cdot)$ is quasiconvex
%\footnote{a finite valued quasiconvex function is locally Lipschitz, hence continuous and locally bounded, thus the class of test functions  can be replaced by $C^\infty_c$ or $C^1_0$ as in Mingione-Mucci}  
for $\mathcal{L}^N$-a.e. $x \in \Omega$;
\item [-]  there exists $C_2>0$ such that 
 %satisfying
\begin{equation}\label{growthWxab1}
W(x,  \xi) \le \,  C_2( |\xi|+1), \qquad \hbox{for $\mathcal{L}^N$-a.e. $x \in \Omega$, for every $\xi \in \mathbb R^{d \times N}.$}
\end{equation} 
 
\end{itemize} Let  $p_0\in L^\infty(\Omega)$ such that  $p_0\in \mathcal P_b^{log}(A)$ for every $A\Subset \Omega$. Let $\{p_k\}\subseteq L^\infty(\Omega)$
 and $\{v_k\}\subset W^{1,1}(\Omega;\mathbb R^d)$
such  that $ \ p_k\to p_0$ in $L^{\infty}(\Omega)$ and $v_k\wto v$ in $L^{1}(\Omega, \R^d)$  as $k \to + \infty$.
If  $\{v_k\}$ satisfies 
\begin{equation}\label{equilim}
\sup_{k\in\N} \int_\Omega |\nabla v_k(x)|^{p_k(x)}dx < +\infty,
\end{equation} then \begin{equation} \label{reglimv0}
v \in W^{1,p_0(\cdot)}_{\rm loc}(\Omega; \mathbb{R}^d),\,
 |\nabla v|\in L^{p_0(\cdot)} (\Omega),
 \end{equation} and  $$ \int_{\Omega}W^{p_0(x)}(x,\nabla v(x))dx\leq \liminf_{k\to +\infty} \int_{\Omega} W^{p_k(x)}(x,\nabla v_k(x))dx.$$

\end{thm}

 \begin{thm}\label{semicon2} 
Let $\Omega\subseteq \R^N$ be an open set.   Let $W: \R^{d\times N}\to [0,+\infty)$ be a  quasiconvex function. Assume that 
there exists a constant  $ C_2>0$ such that 
\beq\label{Wlingrowth}
W(\xi) \le \, C_2( |\xi|+1),\qquad\hbox{for every $\xi \in \mathbb R^{d\times N}$}.
\eeq

Let  $p_0$, $\{p_k\}$, $v$ and $\{v_k\}$ be as in Theorem {\rm \ref{semicon}}. % \in C(\Omega)\cap L^\infty(\Omega)$ such that $p_0\in \mathcal P_b^{log}(A)$ for every $A\Subset \Omega$. 
Let $a_0 \in L^\infty(\Omega)$ be such that  $a_0\geq 0$ $\mathcal{L}^N$-a.e. in $ \Omega$ %let $\{p_k \}, \{a_k\} \subseteq L^{\infty}(\Omega)$ 
 and assume that $a_k \to a_0$ %and $p_k \to p_0$ 
 in $L^{\infty}(\Omega)$ as $k \to + \infty$.  %Let  $\{v_k \}\subset W^{1,1}(\Omega;\mathbb R^d)$ such that $v_k\wto v$ in $L^{1}(\Omega, \R^d)$ as $k \to +\infty$ and  satisfies  \eqref{equilim}.
%\sup_k \int_{\Omega} |\nabla v_k(x)|^{p_k(x)} \, dx < + \infty,
%\eeq 
 Then $v$ satisfies \eqref{reglimv0} and 
 %\in W^{1,p_0(\cdot)}_{\rm loc}(\Omega; \mathbb{R}^d)$, $ |\nabla v|\in L^{p_0(\cdot)} (\Omega)$ }  
 $$\int_{\Omega} a_0(x) W^{p_0(x)}(\nabla v(x))dx\leq \liminf_{k \to +\infty} \int_{\Omega} a_k(x) W^{p_k(x)}(\nabla v_k(x))dx.$$

\end{thm}

 \vspace{2mm}
 Note that, in the previous semicontinuity theorems,  assumption \eqref{equilim} can be dropped  when the energy density $W$ satisfies a suitable coercivity assumption, see Remark \ref{app} (2).
 
 %is coercive, i.e. there exists  a constant $C_1>0$ such that
%$$
%	W(x,\xi) \geq C_1|\xi|-  %\frac{1}{C_1}.
%$$

The asymptotic analysis (via $\Gamma$-convergence) as $\varepsilon \to 0^+$ of the functionals appearing in \eqref{genminpbresc}, in the  $3d-2d$ dimensional reduction hyperelastic setting, besides our study can be carried out analogously for any $nd-md$ dimension reduction, ($m<n, m,n \in \mathbb N$). With this aim,  here and in the sequel, we consider $\Omega_\varepsilon$  and $\Omega$ as in \eqref{Omegaepsilon} and \eqref{Omega1}, respectively. 
The  standard scaling argument in the $x_3$ direction introduced above,  allows us to rephrase the problem in the fixed domain $\Omega$, setting
$v(x)\equiv v(x_\alpha, x_3):= u(x_\alpha, \varepsilon x_3)$, and taking $p_\varepsilon$ as in \eqref{pepsi}. 
Moreover, assuming that the sample is clamped on its lateral boundary, we can make precise the functional space $X_\varepsilon(\Omega)$ appearing in \eqref{genminpbresc}. 
Indeed, denoting the lateral boundary of $\Omega$ by $\displaystyle\partial_L\Omega:= \partial \omega \times \left (-1/2, 1/2\right)$, the functional space  $X_\varepsilon(\Omega;\mathbb R^3)$ can be specialized as   
%\begin{align}\label{latW1px}
\[
W^{1,p_\varepsilon(\cdot)}_L(\Omega;\mathbb R^3) := W^{1,p^-}_{0, \partial_L \Omega} (\Omega;\mathbb R^3) \cap  W^{1,p_\varepsilon(\cdot)}( \Omega ;\mathbb R^3),
\]
%\end{align} 
where, according to Definition \ref{W1r-traccia} below, $W^{1,p^-}_{0, \partial_L \Omega} (\Omega;\mathbb R^3)$ is  the closure in $W^{1,p^-}(\Omega)$ of the subspace 
$$\{v \in W_ {loc} ^{1,p^-}(  \R^3 ;\mathbb R^3):v(x)\equiv 0 \hbox{ in a neighbourhood of } \partial_L\Omega \}.
$$
Consequently, the functional $\mathcal{F}_\varepsilon:L^1(\Omega;\mathbb R^3) \to [0,+\infty]$  in \eqref{genminpbresc} can be rewritten as  
\begin{align}\label{Jeps}
\mathcal{F}_\varepsilon(u):=\left\{
	\begin{array}{ll} \displaystyle \int_\Omega f(x_\alpha, \varepsilon x_3, p_{\varepsilon}(x),\nabla_\alpha u(x),\tfrac{1}{\varepsilon} \nabla_3 u(x))dx, &\hbox{ if }u \in W^{1,p_\varepsilon(\cdot)}_L(\Omega;\mathbb R^3), \\
		\\
		+\infty &\hbox{ otherwise in $L^1(\Omega; \mathbb{R}^3)$}% L^{p(x_\alpha,\varepsilon x_3)}(\Omega;\mathbb R^3),
	\end{array}
	\right.
\end{align}
 where $f_{\varepsilon}$ appearing in \eqref{genminpbresc} is now defined by \eqref{fgen} and $p_{\varepsilon}$ represents the  growth condition of $f_{\varepsilon}$  according to the rescaling \eqref{pepsi}.
 In the following theorems  we aim to represent the $\Gamma$-limit of  the family  $\{\mathcal{F}_\varepsilon\}$ (with respect to the $L^1$- convergence of the deformation fields $u_{\varepsilon}$) under a suitable convergence assumption of $p_{\varepsilon}.$

First we consider the case when $f=f(x_{\alpha}, y, q, \xi)$  is  a function  defined on $\omega\times  (-1/2, 1/2)\times  [1,+\infty)\times \mathbb R^{3\times 3}$,  convex in  the gradient variable  $\xi$ and  depending on the variable $q$ (which assumes the value $p_{\eps}(x)$). %  replacing the function $W_\varepsilon$ in \eqref{genminpb} with suitable other functions. 

% and $\nabla_\alpha v$ denotes the gradient of $v$ with respect to the planar variables $x_\alpha$.

The first result in the dimension reduction framework is the following:

\begin{thm}\label{dimredconv}  Let $\omega \subset \mathbb R^2$ be a bounded, connected,  open set with Lipschitz  boundary.
Let $\displaystyle f:\omega\times  (-1/2, 1/2)\times  [1,+\infty)\times \mathbb R^{3\times 3}\to [0,+\infty)$  be a function such that
\begin{itemize}
\item[-] $f(x_{\alpha},\cdot, \cdot, \cdot)$ is continuous for $\mathcal L^2$-a.e. $x_{\alpha}\in \omega$,
\item[-] $f(\cdot, y, q, \xi)$ is measurable for every $\displaystyle (y,q,\xi) \in  (-1/2, 1/2)\times [1,+\infty) \times 
\mathbb R^{3 \times 3}$,
\item[-] $f(x_\alpha, y,q, \cdot)$ is convex for $\mathcal L^2$-a.e. $x_{\alpha}\in \omega$  and for every $\displaystyle (y,q)\in  (-1/2, 1/2)\times [1,+\infty)$.

\end{itemize}
Assume that  there exist $0<C_1\leq C_2$ such that  
\begin{equation}\label{Vgrowthintro}
	C_1|\xi|^q- \frac 1 {C_1}  \leq f (x_{\alpha},y,  q,  \xi)\leq C_2 (|\xi|^{q} +1) \qquad \end{equation}
for $\mathcal L^2$-a.e. $x_{\alpha}\in \omega$ 
and for every $ \displaystyle (y,q,\xi)\in  (-1/2, 1/2)\times [1, + \infty )\times \mathbb R^{3 \times 3}$.

Let $p \in \mathcal{P}_b(\Omega)$ with $\Omega=\omega\times (-1/2, 1/2)$,  let $p_\varepsilon$ be as in \eqref{pepsi}  and let $\{\mathcal{F}_\varepsilon\}$ be  the family of functionals in  \eqref{Jeps}.  If there exists $p_0\in \mathcal{P}^{log}_b(\omega)$ such that $p_\varepsilon\to p_0 \hbox{ in } L^1(\Omega)$  as $\varepsilon\to 0^+$, then \[
 \Gamma(L^1)\hbox{-}\lim_{\varepsilon \rightarrow 0^+} {\mathcal F}_\varepsilon = {\mathcal F}
 \]   where 
 $\mathcal{F}:L^1(\Omega;\mathbb R^3)\to [0,+\infty]$ is  the functional defined by \begin{equation}\label{Jdef}
 {\mathcal F}(u):=\left\{
 	\begin{array}{ll} \displaystyle \int_\omega {f}_0(x_\alpha, 0, p_0(x_{\alpha}), \nabla_\alpha u(x_\alpha))dx_\alpha , &\hbox{ if }u \in  W^{1,p_0(\cdot)}_0(\omega;\mathbb R^3), \\
	\\
 		+\infty &\hbox{ otherwise in }L^{1}(\Omega;\mathbb R^3),
 	\end{array}
 	\right.
 \end{equation} 
with   $\displaystyle f_0:\omega\times  (-1/2, 1/2)\times \R^{3\times 2}\to \R $   defined by 
\begin{equation}\label{f0ours}f_0(x_\alpha,y, q, \xi_\alpha):= \inf_{\xi_3 \in \mathbb R^3}f(x_\alpha,y, q, \xi_\alpha,\xi_3).
\end{equation}
  \end{thm}

 In the following results, the convexity assumption on $f$ with respect to the gradient variable is dropped, but a special structure is imposed on $f$ in \eqref{fgen} and a stronger convergence assumption on $\{p_{\varepsilon}\}$ is required. More precisely, we neglect the dependence of $f$ on the transverse variable  and the dependence on the variable exponent $p_\varepsilon$ is explicitly of power-law type.
We denote this particular class   of functionals  as 
  \begin{equation}\label{Jepsf}
	\mathcal I_\varepsilon(u):=\left\{
	\begin{array}{ll} \displaystyle \int_\Omega W^{p_{\eps}(x) }\left(x_	\alpha,\nabla_\alpha u(x), \tfrac{1}{\varepsilon}\nabla_3 u(x)\right)dx &\hbox{ if } u \in W_L^{1,p_{\varepsilon}(\cdot)} (\Omega;\mathbb R^3), \\
		\\
		+\infty &\hbox{ otherwise in }L^1(\Omega;\mathbb R^3).
	\end{array}
	\right.
\end{equation}

  \begin{thm} \label{quasiconvex} Let $\omega \subset \mathbb R^2$ be a bounded, connected,  open set with Lipschitz  boundary.
  Let $W:\omega\times \R^{3\times 3}\to [0,+\infty)$ be a  function such that 
 \begin{itemize}
 \item [-] $W(\cdot,\xi)$ is measurable for every $\xi \in \mathbb R^{3 \times 3}$; 
\item [-]  $W(x_{\alpha}, \cdot)$ is continuous  for $\mathcal L^2$-a.e. $x_{\alpha} \in \omega$;
\item [-]  there exist $0<C_1\leq C_2$ such that 
\begin{equation}\label{Wcresc}
 C_1 |\xi|-\frac 1 {C_1} \le \,W(x_\alpha, \xi)  \le \,  C_2( |\xi|+1) \quad  \hbox{ for }\mathcal{L}^2\hbox{-a.e. } x_{\alpha}\in \omega, \hbox{ for every } \xi\in \R^{3\times 3}.
\end{equation}
%\item[-] there exists a  function $w: \omega\times [0,+\infty) \to [0,+\infty) $ such that $w(\cdot, t)$ is measurable for every $t\in [0,+\infty)$ and such that
%  $w(x_{\alpha},\cdot)$ is a   continuous, increasing function for a.e. $x_{\alpha}\in \omega$, satisfying  $\lim_{t\to 0^+}w(x_{\alpha},t)=0$ with 
%\begin{equation}\label{unifcont}|W(x_{\alpha},\xi)-W(x_{\alpha},\eta)|\leq  w(x_{\alpha}, |\xi-\eta|)\hbox{ for every } \xi, \eta\in \R^{3\times 3} \qquad \hbox{for } \mathcal L^2\hbox{-a.e. } x_{\alpha}\in \omega.\end{equation}
   \end{itemize}

Assume that \begin{equation}\label{unifcont}|W(x_{\alpha},\xi)-W(x_{\alpha},\eta)|\leq  w(x_{\alpha}, |\xi-\eta|)\hbox{ for every } \xi, \eta\in \R^{3\times 3} \qquad \hbox{for } \mathcal L^2\hbox{-a.e. } x_{\alpha}\in \omega,\end{equation} where 
 $w: \omega\times [0,+\infty) \to [0,+\infty) $ is such that
 \begin{itemize} 
 \item[-] $w(\cdot, t)$ is measurable for every $t\in [0,+\infty)$, 
 \item[-]  $w(x_{\alpha},\cdot)$ is a   continuous, increasing function for ${\mathcal L }^2$ a.e. $x_{\alpha}\in \omega$ satisfying $\lim_{t\to 0^+}w(x_{\alpha},t)=0$.
\end{itemize}
   Let $p \in \mathcal{P}_b(\Omega)$ with $\Omega=\omega\times (-1/2, 1/2)$,  let $p_\varepsilon$ be as in \eqref{pepsi}  and let $\{\mathcal{I}_\varepsilon\}$ be  the family of functionals in  \eqref{Jepsf}.  If there exists $p_0\in \mathcal{P}^{log}_b(\omega)$ such that $p_\varepsilon\to p_0 \hbox{ in } L^\infty(\Omega)$  as $\varepsilon\to 0^+$, then 

$$\Gamma(L^1)\hbox{-}\lim_{\varepsilon \to 0^+} \mathcal I_{\varepsilon}=\mathcal I$$ 

where $\mathcal{I}:L^{1}(\Omega;\mathbb R^3)\to [0,+\infty]$ is the functional given  by 
  
 \begin{equation}\label{Jqc}
\mathcal I(u):=\left\{
\begin{array}{ll} \displaystyle \int_\omega {Q( W_0^{p_0(x_{\alpha})}})(x_\alpha, \nabla_\alpha u(x_{\alpha}))dx_\alpha , &\hbox{ if }u \in W_0^{1,p_0(\cdot)}(\omega;\R^3), \\
\\
	+\infty &\hbox{ otherwise in }L^{1}(\Omega;\mathbb R^3)
	\end{array}
\right.
\end{equation}
with   $W_0:\omega\times \R^{3\times 2}\to \R $ defined as %\begin{equation}\label{f0gen}
\[
	W_0(x_\alpha, \xi_\alpha):= \inf_{\xi_3 \in \mathbb R^3}W(x_\alpha,  \xi_\alpha,\xi_3),
 \]
%\end{equation}
for $\mathcal L^2$-a.e. $x_\alpha \in \omega$ and for every $\xi_\alpha \in \mathbb R^{3 \times 2}$, and 
$Q(W_0^{p_0(x_\alpha)})(x_\alpha,\cdot)$ denotes the quasiconvex envelope of $W_0^{p_0(x_\alpha)}(x_\alpha,\cdot)$, namely the greatest quasiconvex function below $W_0^{p_0(x_\alpha)}(x_\alpha,\cdot)$ for $\mathcal L^2$-a.e. $x_\alpha \in \omega$.

 \end{thm}

The third model considered in $3d-2d$-dimensional reduction setting is linked to the class of functionals considered in  Theorem \ref{semicon2}, i.e. the function $f$ in \eqref{fgen}  still does not depend on the transverse variable, and the point dependence on the planar variable $x_\alpha$ is of product type, while the dependence on the variable exponents is of explicit power type.  Its proof, relying on the technical Lemma \ref{lemtec}, is not a direct application of Theorem \ref{semicon2}, which, indeed, may provide only a possibly strict  lower bound.

More precisely,  the  family of functionals we consider is the following

  \begin{equation}\label{Jepsfa(x)}
	\mathcal J_\varepsilon(u):=\left\{
	\begin{array}{ll} \displaystyle \int_\Omega a(x_\alpha, \varepsilon x_3)
 W^{p_\varepsilon(x) }\left(\nabla_\alpha u(x), \tfrac{1}{\varepsilon}\nabla_3 u(x)\right)dx&\hbox{ if }u \in W_L^{1,p_{\varepsilon}(\cdot)} (\Omega;\mathbb R^3), \\
		\\
		+\infty &\hbox{ otherwise in }L^1(\Omega;\mathbb R^3),
	\end{array}
	\right.
\end{equation}
where $a \in L^\infty(\Omega)$.

%{\color{red} Da controllare gli enunciati quando si ricontrolla la proof sia sopra sia sotto; controllare come richiamiamo le crescite}

 \begin{thm} \label{ubqcxdimred}
		Let $W: \R^{3\times 3}\to [0,+\infty)$ be a uniformly continuous function such that there exist $C_1>0$ and $C_2\geq 1$ with
\begin{equation}\label{coerca(x)}
 C_1 |\xi|-\frac 1 {C_1} \le \,W( \xi)  \le \,  C_2( |\xi|+1)   \hbox{ for every } \xi\in \R^{3\times 3}.
\end{equation}

Let $a\in L^{\infty}(\Omega)$ be such that 
  \begin{align}\label{a-}
  a^-:=\inf_{x \in \Omega} a(x) >0,
  \end{align} and let 
  %\begin{align}\label{aepsi}
  \[
a_\varepsilon(x):= a(x_\alpha, \varepsilon x_3).
\]
%\end{align} 
Assume that there exists  $a_0\in L^\infty(\omega)$ such that 
  \begin{align}\label{a0def1}\lim_{\varepsilon \to 0^+} a_\varepsilon(x)=\lim_{\varepsilon \to 0^+}a(x_\alpha, \varepsilon x_3) = a_0(x_\alpha) \hbox{ in }L^\infty(\Omega).
  \end{align}
 Let $p \in \mathcal{P}_b(\Omega)$,  let  $p_\varepsilon$ be as in \eqref{pepsi} and let $\{\mathcal{J}_\varepsilon\}$ be  the family of functionals in  \eqref{Jepsfa(x)}. If there exists $p_0\in \mathcal{P}^{log}_b(\omega)$ such that $p_\varepsilon\to p_0$ in $L^\infty(\Omega) $  as $\varepsilon\to 0^+$,
   then 
$$ \Gamma(L^1)\hbox{-}\lim_{\varepsilon\to 0^+} \mathcal J_{\varepsilon}=\mathcal J
		$$
where $\mathcal{J}:L^{1}(\Omega;\mathbb R^3)\to [0,+\infty]$ is the functional given  by

\begin{equation}\label{Jqca(x)}
\mathcal J(v):=\left\{
\begin{array}{ll}\displaystyle \int_\omega a_0(x_{\alpha}) Q( W_0^{p_0(x_{\alpha})})(\nabla_\alpha v(x_{\alpha}))dx_\alpha , &\hbox{ if }v \in W_0^{1,p_0(\cdot)}(\omega;\R^3), \\
\\
	+\infty &\hbox{ otherwise in }L^{1}(\Omega;\mathbb R^3),
	\end{array}
\right.
\end{equation}
with $W_0:\mathbb R^{3 \times 2}\to [0,+\infty)$ defined by
\begin{align}\label{QW0nox}
W_0(\xi_\alpha):= \inf_{\xi_3 \in \mathbb R^3}W(\xi_\alpha, \xi_3),
\end{align}
and
	 $ Q(W_0^{p_0(x_\alpha)})$ denoting the quasiconvex envelope of $W_0^{p_0(x_\alpha)}.$
	 %there exists $\eps_k\to 0$ and $(u_{{\eps_k}})\subseteq W^{1,p_{\varepsilon_k}(x)}_{L }(\Omega;\mathbb R^3)$ 
	%such that $u_{\eps_k} \to u $ in $L^{1}(\Omega;\mathbb R^3)$, with 
	%$u\in W^{1,{p_0(\cdot)}}_0(\omega;\mathbb R^3)$ 
	%and 
	
\end{thm}

 %Moreover we assume that $W$ and $p$ satisfy the following assumptions
%\begin{align}
%	\label{H1} \exists C, C' >0 \hbox{ such that }
%	\frac{1}{C} |\xi|^{p(x)}- C'\leq W(x,\xi)\leq C(|\xi|^{p(x)}+1),
%\end{align} 

%\textcolor{OliveGreen} $p_0$ limite puntuale di $p_\varepsilon$ \'e in ${\mathcal P}^{log}(\omega)$, dal momento che \eqref{p-p+} implicano $p_0^- >1$ e $p_0^+ <+\infty$. 
We note that,  in the  dimensional reduction theorems,  thanks to the special structure of the sequence $\{p_\varepsilon\}$ defined by \eqref{pepsi},  if   $p \in \mathcal P_b^{log}(\Omega)$,  then the family 
%$\{p_\varepsilon\}_{\varepsilon < \varepsilon_0}$  
$\{p_\varepsilon\}$
is   equicontinuous on $\overline{\Omega}$ and  equibounded since, for every $ \eps>0,$ it holds

\begin{equation}\label{p-g2} p^+_{\eps}:=\supess_{\omega\times \left(-\frac {1}{2},\frac 1{2} \right) } p(x_\alpha,\varepsilon x_3) = \supess_{\omega\times \left(-\frac {\eps}{2},\frac {\eps}{2} \right) } p(x_\alpha, x_3)\leq p^+ <+\infty  .\end{equation} 
and  \begin{equation}\label{p-g1} p^-_{\eps}:=\infess_{\omega\times \left(-\frac {1}{2},\frac 1{2} \right) } p(x_\alpha,\varepsilon x_3) = \infess_{\omega\times \left(-\frac {\eps}{2},\frac {\eps}{2} \right) } p(x_\alpha, x_3)\geq p^-  >1 .\end{equation}

Hence,  thanks to the Ascoli-Arzelà Theorem,  the  sequence    $p_{\eps}\to p_0$ in $L^\infty(\Omega)$ as $\varepsilon\to 0^+$ where $p_0(x_\alpha)=p(x_{\alpha},0)\in \mathcal P_b^{log}(\Omega) $.
\vspace{2mm}

It is worth observing that the proofs of Theorems \ref{quasiconvex} and \ref{ubqcxdimred}  
 do not follow as mere applications of Theorems \ref{semicon} and \ref{semicon2}, indeed the representation results obtained in \eqref{Jqc} and \eqref{Jqca(x)} involve densities of the type $Q(W_0^{p_0(x_\alpha)})(\cdot)$ which are, in general, greater than $(QW_0)^{p_0(x_\alpha)}(\cdot)$.  This is easily seen in the scalar case (where the quasiconvex envelope coincides with the convex one): for instance, when $p>1$ is constant and $W=W_0$ is a homogeneous function defined as $W(\xi)=\max\{|\xi|, |\xi|^{1/p}\}$
it results
 $$Q(W^p)(\xi)= W^p(\xi)= \max\{|\xi|^p,|\xi|\}\geq |\xi|^p=(QW)^p(\xi), \hbox{ for every } \xi \in \mathbb R,$$
 where the inequality is strict in $[-1,1]$.
\vspace{3 mm}

The paper is organized as follows: Section \ref{Not&pre} contains notation and preliminary results regarding the functional spaces and $\Gamma$-convergence. Section \ref{SCIres} is devoted to the proofs of Theorems \ref{semicon} and \ref{semicon2} which follow as corollaries of the more general result Theorem \ref{thm3.7}. Finally Section \ref{secdimred} contains the proofs of the dimensional reduction results stated above, {namely Theorems \ref{dimredconv}, \ref{quasiconvex} and \ref{ubqcxdimred}, together with a compactness result for energy bounded sequences in $L^1(\Omega),$ namely Proposition \ref{compactness}, which motivates the choice of the topology for our results.}
\color{black}

\section{ Notation and preliminary results}\label{Not&pre}

In the sequel $\Omega$ indicates a generic open set of $\mathbb{R}^N,  N \ge 1$; by $\mathcal A(\Omega)$ we denote the class of open subsets of $\Omega$ and by $\mathcal A_0(\Omega)$ we denote the subclass of $\mathcal A(\Omega)$ whose elements are well contained in $\Omega$, i.e. $B \in \mathcal{A}_0(\Omega) $ if $B \Subset \Omega$.
We denote by $\mathcal{L}^N(\Omega)$ the $N$-dimensional Lebesgue measure of the set $\Omega$.

\subsection{Variable exponents Lebesgue spaces}
In  this section we collect some basic results concerning variable exponent Lebesgue spaces. For more details we refer to the monograph \cite{DHHR11}, see also \cite{KR91}, \cite{ELN99}, \cite{ER1}, \cite{ER2}.
\\
Let  $\Omega \subset \mathbb{R}^N$ be  an open set (where $N\geq 1$). 

\begin{defn}\label{p} For any (Lebesgue) measurable function $p: \Omega \rightarrow [1, + \infty]$
% (invece di  $p: \overline{\Omega} \rightarrow (1, + \infty)$) 
we define
$$
p^- := \displaystyle \infess_{x \in \Omega} p(x) \qquad \qquad p^+ :=\displaystyle
\supess_{x \in \Omega} p(x).
$$
Such function $p$ is called {\it variable exponent} on $\Omega$.  If $ p^+<+\infty$  then we call $p$ a {\it bounded variable exponent}. 
\\
We denote by $\mathcal P(\Omega)$ the class of variable exponents and with $\mathcal{P}_b(\Omega)$ the class of variable exponents satisfying \eqref{p-p+}.
% Moreover, we denote by $\mathcal P^{log}(\Omega)$ the class of variable exponents $p$ satisfying \eqref{logp} \textcolor{blue}{attenzione: e' senza la decay} and by $\mathcal P^{log}_b(\Omega)$ the subset of $\mathcal P^{log}(\Omega)$ whose elements satisfy \eqref{p-p+}. 
\end{defn}

%In the sequel we need to introduce the Lebesgue spaces with variable exponents, $L^{p(\cdot)}(\Omega)$. They differ from the classical $L^p$ spaces because now the exponent $p$ is not constant but it is a variable exponent in the sense specified above. Originally the spaces $L^{p(\cdot)}$  have been introduced in the case $1 \le p^- \le p^+ <+ \infty$ by Orlicz \cite{O31} in 1931 and,  in the case $p^+ = \infty$,   by Sharpudinov \cite{S79} and later  (in the higher dimensional case), by Kov\'a\v{c}ik and R\'akosn\'ik \cite{KR91}.

In the sequel we consider the case $p^+ < +\infty$. In this case, the variable exponent Lebesgue space $L^{p(\cdot)}(\Omega)$ can be defined as \[
L^{p(\cdot)}(\Omega) := \left \{u: \Omega \rightarrow \R \,\,\, \textnormal{measurable such that} \,\, \int_{\Omega} |u(x)|^{p(x)} \, dx < + \infty   \right \}.
\]
%\textcolor{magenta}{semplificare!!!}

Let us note that in the case  $p^+ = +\infty$ the  space above defined  may even fail to be a vector space (see \cite{CM} Section 2)
and a different definition of the variable Lebesgue spaces  has been  given in order to preserve the vectorial structure of the space (we refer to \cite{DHHR11}, Definition 3.2.1).
In addition,  if $p^+ < + \infty,$ then $L^{p(\cdot)}(\Omega)$ is a Banach space endowed with the {\it Luxemburg norm} 
\[
\|u\|_{p(\cdot)} := \inf \left \{\lambda  > 0: \,\,\, \int_{\Omega} \left |\frac{u(x)}{\lambda} \right |^{p(x)} \, dx \le \, 1  \right \}
\]
(see  Theorem 3.2.7 in  \cite{DHHR11}). Moreover,  if    $p^+ < + \infty$,     the space $L^{p(\cdot)}(\Omega)$ is separable and $\mathcal{C}^{\infty}_0(\Omega) \hbox{ is dense in }L^{p(\cdot)}(\Omega)$, while, if $1 < p^- \le p^+ < + \infty$,   the space $L^{p(\cdot)}(\Omega)$ is reflexive and uniformly convex (see 
Theorem 3.4.12, Theorem 3.4.7, Theorem 3.4.9 and Theorem 3.4.12  in \cite{DHHR11}). For any variable exponent $p$, we define $p'$ by setting
\[
\frac{1}{p(x)} + \frac{1}{p'(x)} = 1,
\]
with the convention that, if $p(x) =+ \infty$ then $p'(x) = 1$. The function $p'$ is called {\it the dual variable exponent of $p$}. 
%For any $p(x)\in [1, N)$, we denote, as in the case of constant exponents, by $p^*(x)$, the variable exponent defined as
%$
	%\frac{1}{p^*(x)}:= \frac{1}{p(x)}-\frac{1}{N}.$
%	\end{equation}

The following result  holds (for more details,  see \cite[Lemma 3.2.20]{DHHR11} in the case  $\varphi(t)=|t|^p$).

\begin{thm} {\sl (H\"older's inequality)}
Let $p,q,s$ be measurable exponents such that
\[
\frac{1}{s(x)} = \frac{1}{p(x)} + \frac{1}{q(x)} \qquad
 \hbox{for } \mathcal L^N\hbox{-a.e. } x\in \Omega. \]
Then, for all $f \in L^{p(\cdot)}(\Omega)$ and $g \in L^{q(\cdot)}(\Omega)$, it holds 
\[
\|fg\|_{s(\cdot)} \le \, \left( \left(\frac{s}p\right)^+ + \left(\frac{s}q\right)^+ \right)\, \|f\|_{p(\cdot)} \, \|g\|_{q(\cdot)}
\]
 where, in the case $s = p = q = \infty$, we use the convention $\frac{s}{p} = \frac{s}{q} = 1.$ In particular, in the case $s  = 1$, it holds
\[
\left | \int_{\Omega} f \, g \, dx\right| \leq  2 \|f\|_{p(\cdot)} \, \|g\|_{p'(\cdot)}\ .
\] 
\end{thm}

We  introduce the {\it modular} of the space $L^{p(\cdot)}(\Omega)$ which is the mapping $\rho_{p(\cdot)}: L^{p(\cdot)}(\Omega) \rightarrow \mathbb{R}$ defined by
\[
\rho_{p(\cdot)}(u) := \int_{\Omega} |u(x)|^{p(x)} \,dx.
\]
Thanks to Lemma 3.2.4  in \cite{DHHR11}, for every $u \in L^{p(\cdot)}(\Omega)$ 
\begin{align}
 \|u\|_{p(\cdot)}  \leq 1 \Longleftrightarrow \rho_{p(\cdot)}(u)\leq  1 \nonumber\\
\|u\|_{p(\cdot)}  \leq 1 \Longrightarrow \rho_{p(\cdot)}(u)\leq    \|u\|_{p(\cdot)}\hbox{ and }
\|u\|_{p(\cdot)}  > 1 \Longrightarrow \|u\|_{p(\cdot)}\le \, \rho_{p(\cdot)}(u)  \label{mod-3}.
\end{align}

The following further results hold in the special case $p^+<+\infty$.
By Lemma 3.2.5 in \cite{DHHR11},    for every $u \in L^{p(\cdot)}(\Omega)$ it holds
\begin{equation} \label{relaztotale}\min \left\{\big( \rho_{p(\cdot)}(u)\big)^{\frac 1 {p^-} } , \big(\rho_{p(\cdot)}(u)\big)^{\frac 1 {p^+} }  \right \} \leq  \|u\|_{p(\cdot)} \leq \max  \left \{  \big( \rho_{p(\cdot)}(u)\big)^{\frac 1 {p^-} } , \big(\rho_{p(\cdot)}(u)\big)^{\frac 1 {p^+} } \right \}.
\end{equation}
In particular,
\beq \min \left \{\big(\L^N(\Omega) \big)^{\frac 1 {p^-} } , \big(\L^N(\Omega)\big)^{\frac 1 {p^+} }   \right \}\leq \|1\|_{p(\cdot)} \leq \max \left\{\big(\L^N(\Omega) \big)^{\frac 1 {p^-} } , \big(\L^N(\Omega)\big)^{\frac 1 {p^+} }  \right \} . \label{mod4}\eeq
Moroever, from \eqref{relaztotale}, taking into account   \eqref{mod-3}, it follows that for every $u \in L^{p(\cdot)}(\Omega)$  \begin{align*}
\|u\|_{p(\cdot)}  > 1 \Longrightarrow \|u\|_{p(\cdot)}^{p^-} \le \, \rho_{p(\cdot)}(u) \le \, \|u\|^{p^+}_{p(\cdot)} \hbox{  and }\|u\|_{p(\cdot)} < 1 \Longrightarrow \|u\|_{p(\cdot)}^{p^+} \le \, \rho_{p(\cdot)}(u) \le \, \|u\|^{p^-}_{p(\cdot)}.
\end{align*}

Finally, by 
\cite[Corollary 3.3.4]{DHHR11}, if $0 <  \mathcal{L}^N(\Omega) < + \infty$ and $p$ and $q$ are variable exponents such that $q \le p$   $\mathcal L^N$ a.e. in $\Omega$, then the embedding $L^{p(\cdot)}(\Omega) \hookrightarrow L^{q(\cdot)}(\Omega)$ is continuous. 
%Set
 %$$\frac 1  {r(x)}=\frac 1 {q(x)}-\frac 1 {p(x)};$$ 
In view of \cite[Theorem 3.3.1, part(a)]{DHHR11} and  \eqref{mod4},
   the embedding constant $C_{p,q,\Om}$ satisfies

\begin{align}\label{cpqomega} C_{p,q,\Om}%&\leq 2 ||1_{\Omega}||_{L^{r(\cdot)}(\Omega)} 
\leq 
2  \min\left \{\mathcal{L}^N(\Omega)  + 1,
 \max  \Big \{  {\mathcal{L}^N(\Omega)}^{{(\frac 1 q-\frac 1 p)}^+},  {\mathcal{L}^N(\Omega)}^{{(\frac 1 q-\frac 1 p)}^-}  \Big \}\right\}\leq 2 ( \mathcal{L}^N(\Omega)+1) .\end{align}

 In particular, if $q(x)\equiv p^-$ 
%$$\supess_{\Omega}\left( \frac 1 {p^-}-\frac 1 {p(x)}\right)=\frac 1 {p^-}-\frac 1 {p^+},\qquad \infess_{\Omega}\left(  -\frac 1 {p^-}+\frac 1 {p(x)} \right)=0$$  that imply 
$$ C_{p,p^-,\Om}\leq 2  \min\left \{  \mathcal{L}^N(\Omega), {\mathcal{L}^N(\Omega)}^{\frac 1 {p^-}-\frac 1 {p^+}}
\right\} +2.$$

 \subsection{Trace operator in Sobolev spaces}
 
In view of introducing some important results concerning variable exponents Sobolev spaces, in this section we recall some useful remarks concerning the trace operator. 

 Let  $\Omega$ be a bounded, connected open set with Lipschitz boundary and $1\leq r\leq \infty$.
Hence, set 
$$s=\begin{cases} \frac{(N-1)r}{N-r} & \hbox{ if } r<N\\
q \in [1,+\infty) & \hbox{ if } r=N\\
+\infty  & \hbox{ if  } r>N, 
\end{cases}
$$
it is well defined the  linear and continuous ``trace operator''  $T_{\partial \Omega}: W^{1,r}(\Omega) \to  L^s(\partial \Omega)$, where the measure on $\partial \Omega$ is ${\mathcal H}^{N-1}$, such that 
\begin{description}
\item{(i)} when  $1\leq r\leq N$,  it holds $$T_{\partial \Omega}(u)(x)=u(x) \hbox{ for $\mathcal H^{N-1}$ a.e.   } x\in \partial \Omega,  \hbox{ for every } u\in W^{1,r}(\Omega)\cap C(\overline{\Omega});$$
\item{(ii)}  when  $r>N$,  it holds $$T_{\partial \Omega}(u)(x)=u(x) \hbox{ for every   } x\in \partial \Omega, \hbox{ for every } u\in W^{1,r}(\Omega). $$
\end{description}
(see   \cite[Theorem 4.3.12]{CDEA}).

\begin{defn}
\label{W1r-traccia} Let $\mathcal{B}(\partial \Omega)$ be the class of Borel subsets of $\partial \Omega$. 
For every set $\Gamma\in \mathcal{B}(\partial \Omega)$ and $1\leq r<+\infty$,  we define
$W^{1,r}_{0,\Gamma}(\Omega)$ to be the closure in $W^{1,r}(\Omega)$ of the subspace 
\[
\{u\in W_{loc}^{1,r}(\R^N): u\equiv 0 \hbox{ on a neighbourhood  of } \Gamma\}.
\]

\end{defn}

Hence  $W^{1,r}_{0,\Gamma}(\Omega)$  is closed once we endow it with the weak-$W^{1,r}(\Omega)$ topology and 
 $$W^{1,r}_{0, \Gamma}(\Omega)\subseteq \{u\in W^{1,r}(\Omega): T_{\partial \Omega}u\equiv 0 \hbox{ for $\mathcal H^{N-1}$-a.e. } x\in \Gamma \}.
$$
Thanks to \cite[ Proposition 4.3.13]{CDEA}  it holds
 %\begin{equation}\label{uguali} 
 \[
 W^{1,r}_{0, \partial \Omega}(\Omega)=W^{1,r}_{0}(\Omega)= \{u\in W^{1,r}(\Omega): T_{\partial \Omega}u\equiv 0 \}.
 \]
%\end{equation}
 Finally,  by \cite[Theorem 4.3.18]{CDEA},  for $1\leq r < + \infty$,  the following  Poincarè inequality holds 

\begin{equation}\label{Poincaretrace}
	\|u\|_{L^{ r^*}(\Omega)}\leq C(\Omega,r,N,\Gamma) \|\nabla u\|_{L^{ r}(\Omega)} \; \; \forall u \in W^{1,r}_{0, \Gamma}(\Omega),
	\end{equation}
where
 \[
r^* = \left \{
\begin{array}{lll}
\!\!\!\!\!\! & \displaystyle \frac{Nr}{N-r} & \qquad \textnormal{if $r \in [1, N)$}\\[2mm]
\!\!\!\!\!\! & \displaystyle + \infty & \qquad \textnormal{otherwise.}
\end{array}
\right.
 \]

\subsection{Variable exponents Sobolev spaces}
In this subsection we recall the definition of  variable exponents Sobolev spaces.  For more details we refer to \cite{CM} (see also \cite{DHHR11}, Definition 8.1.2). 

\begin{defn}

%\label{Sobpvar}
 
Let $k,d\in \N$, $k\geq 0$,  and let  $p\in \mathcal{P}(\Omega).$ We define
$$W^{k, p(\cdot)}(\Omega;\R^d):=\{ u:\Omega\to \R^d: u, \partial_{\alpha} u \in L^{p(\cdot)}(\Omega;\R^d) \quad \forall \alpha \hbox{  multi-index such that $|\alpha| \le \, k$}\},$$
where $$L^{p(\cdot)}(\Omega; \R^d):=\{ u:\Omega\to \R^d :\ |u| \in L^{p(\cdot)}(\Omega)\}.$$
We define the semimodular on $W^{k, p(\cdot)}(\Omega)$ by
\[
\rho_{W^{k, p(\cdot)}(\Omega)}(u) := \sum_{0 \le |\alpha| \le k} \rho_{L^{p(\cdot)}(\Omega)}(|\partial_{\alpha} u|)
\]
which induces a norm  by 
\[
\|u\|_{W^{k, p(\cdot)}(\Omega)} := \inf \left \{ \lambda > 0: \,\, \rho_{W^{k, p(\cdot)}(\Omega)}\left (\frac{u}{\lambda} \right ) \le \, 1 \right \}.
\]
\end{defn}
 For $k \in \mathbb{N} \setminus \{0\},$ the space $W^{k, p(\cdot)}(\Omega)$ is called  {\it Sobolev space} and its elements are called {\it Sobolev functions}. Clearly $W^{0, p(\cdot)}(\Omega) = L^{p(\cdot)}(\Omega)$.
The space  $W^{k, p(\cdot)}(\Omega)$  is a Banach space,
which is separable if $1\leq {p^-}\leq p^+<+\infty,$ and reflexive and uniformly convex if $1<{p^-}\leq p^+<+\infty$ (see \cite[Theorem 8.1.6]{DHHR11}).
We recall that the Sobolev conjugate exponent $p^*:\Omega \to (1,+\infty]$ is defined as
 $$p^*(x):=\begin{cases} \frac {Np(x)}{N-p(x)} \hbox{ if } p(x)<N,\\
	+\infty  \quad \,\hbox{ otherwise.} 
\end{cases}
$$
%Since for every $x\in\Om$ such %that $p(x)<N$, it results %$$\frac {p^- N}{ N-p^- }\leq %\frac {Np(x) }{ N-p(x) },$$
%we observe that $(p^-)^* \leq  %\inf_{x\in \Om} \left\{\frac %{N p(x) }{ N-p(x) }\right\}= %(p^* )^-.$

Finally, we define  $W^{1,p(\cdot)}_0(\Omega)$ as the closure of the set of $W^{1, p(\cdot)}(\Omega)$-functions
with compact support.

Now we complement Definition \ref{p} with the following one.
\begin{defn}\label{pbis} We denote by $\mathcal P^{log}(\Omega)$ the class of variable exponents $p$ satisfying \eqref{logp} and by $\mathcal P^{log}_b(\Omega) = \mathcal P^{log}(\Omega) \cap \mathcal{P}_b(\Omega)$. 
\end{defn}

\begin{rem} We recall that, when $\Omega$ is a bounded open set, \eqref{logp}
%the locally Log-continuity 
implies  the so called {\sl  decay condition}  
\begin{align*}
	%\label{logdecay}
	\exists p_\infty \in \mathbb R, \gamma_\infty > 0 : |p(x)-p_\infty| \leq \frac{\gamma_\infty}{\log (e+|x|)} \hbox{ for every }x\in \Omega,
\end{align*}
(see \cite[Remark 2.4]{DHHMS}).
This ensures that our class  $\mathcal P^{log}(\Omega)$  coincides with the class of variable exponents denoted in the  same way in \cite[Definition 4.1.4]{DHHR11}.
%remark 4.3.10 e corollary 4.3.11 in {D} )
\end{rem}

%\begin{thm}[\cite{DHHR11}Theorem 8.4.2 ] Let $\Omega$ be a  bounded open set. Assume that $p \in \mathcal{P}_{	log}(\Omega)$ Then the inequality
%\[
%\|u\|_{L^{p(\cdot)}(\Omega)} \le \, c \, \|\nabla u\|_{L^{p(\cdot)}(\Omega)},
%\]
%holds for every $u \in W^{1, p(\cdot)}_0(\Omega)$. Here the constant $c$ depends on $p, \Omega, \delta$ and the dimension $n$.}
%\end{thm}
%{\color{magenta} Michela inserisce la remark part (1) dove serve}
In the sequel we frequently  use the following results.
\begin{rem}\label{remdens}  Let  $\Omega$  be a bounded open set. Hence 
\begin{enumerate} 

\item[{\rm (1)}]  if $p \in \mathcal P^{log}(\Omega)$ satisfies $p^+<+\infty$,  then  $C_0^\infty(\Omega)$ is dense in $W^{1,p(\cdot)}_0(\Omega)$ (see \cite[Corollary 11.2.4]{DHHR11}); %\textcolor{blue}  satisfies $p^+<+\infty$,  then  $C_0^\infty(\Omega)$ is dense in $W^{1,p(x)}_0(\Omega)$ (see \cite[Corollary 11.2.4]{DHHR11}); %\textcolor{blue}{vedi Proposizione 4.6 di Mingione-Mucci + Proposizione 5.2 di Mingione-Mucci}
\item[{\rm (2)}]  if $p \in \mathcal P^{log}(\Omega)$,  then 
 the following Poincar\'e inequality holds
\begin{equation}\label{poinc}\|u\|_{L^{ p(\cdot)}(\Omega)}\leq c diam(\Omega) \|\nabla u\|_{L^{ p(\cdot)}(\Omega)} \  \forall u \in W_0^{1,p(\cdot)}(\Omega)\end{equation}
\end{enumerate}
\end{rem}

Finally, we report here the proof of the compact embedding $ W^{1,p(\cdot)}(\Om;\mathbb R^d)\hookrightarrow L^{p(\cdot)}(\Om;\mathbb R^d)$ by using arguments similar to those employed in \cite[Proposition 3.3]{MOSS} (see, also, \cite[Theorem 1.3]{FSZ}).

\begin{prop}\label{embp0x} Let $\Omega \subset \mathbb R^N$ be a Lipschitz bounded  open set and let $p\in C(\overline{\Omega})$ such that $p\geq 1$. Then \begin{equation}\label{embp0xeq}\{ u\in L^1(\Omega;\R^d):\,  \nabla u\in  L^{p(\cdot)}(\Om; \R^{d\times  N})\}=W^{1, p(\cdot)}(\Om;\mathbb R^d).
	\end{equation}
	Moreover, the embedding $ W^{1,p(\cdot)}(\Om;\mathbb R^d)\hookrightarrow L^{p(\cdot)}(\Om;\mathbb R^d)$  is compact and on $W^{1,p(\cdot)}(\Om;\mathbb R^d)$  the norm $$||| u |||= \|u\|_{L^1(\Om;\mathbb R^d)}+\|\nabla u\|_{L^{p(\cdot)}(\Om;\mathbb R^{d\times \N})}$$
 is   equivalent to the norm $\|\cdot \|_{W^{1,p(\cdot)}(\Om;\mathbb R^d)}$.  \end{prop}
\proof  Without loss of generality,  we can assume $d=1$.  Identity \eqref{embp0xeq} follows by \cite[Lemma 2.4]{CM}.
In order to show that the embedding $ W^{1,p(\cdot)}(\Om)\hookrightarrow L^{p(\cdot)}(\Om)$  is compact, first of all we construct a suitable covering of $\overline{\Omega}.$
Since $p$ is uniformly  continuous  on $\overline{\Omega}$, for every fixed $0<\sigma< \frac{(p^-)^2}{N+p^-}$  there exists $\rho>0$ such that  
\begin{equation}\label{sigmaest}
	\sup_{ B_r\cap \Omega} p -\inf _{B_r\cap \Omega} p<\sigma,
\end{equation} for every ball $B_r $ di radius $r<\rho$. 
Being  $\overline{\Omega}$  a compact set  with Lipschitz boundary,  we can find a finite number of  open balls $(B_{j})_{j\in F}$ and  $(\widetilde B_{{j}})_{j\in F}$,  such that $\overline{B_{j}}\subseteq \widetilde B_{j}$, $B_{j}$ and $\widetilde B_{j}$ have the same centers and
\begin{equation}
\label{propball}
    \begin{array}{lll}
&	\overline{ \Om}\subseteq \bigcup_{j\in F} B_{j}, \\
&	\Om\cap  \widetilde B_{j} \hbox{ has Lipschitz continuous boundary,} \\
&	\hbox{for every }j \hbox{ the  radius }r_{j} \hbox{ of }\widetilde B_{j} \hbox{ satisfies }r_j <\rho.
\end{array}
\end{equation}

For every  $j \in F$, %$\tilde B_j$  
we set \beq\label{p0j} p_{j}^+=\sup _{\widetilde B_j\cap \Om} p, \quad p_{j}^-=\inf _{\widetilde B_j\cap \Om} p.\eeq
We  claim that, for every $j\in F$, there exists $\delta_j>0$ such that  the embedding   \beq\label{compE1}   W^{1,p_{0,j}^- -\delta_j}(\Om) \hookrightarrow  L^{p_{0,j}^+}(\Om) \hbox{  holds and is compact}.\eeq
Indeed, fix $j\in F$. We distinguish two cases: 
\begin{itemize}
	\item if $p_{j}^->N$, then there exists $\delta_j>0$ such that $p_{j}^--\delta_j>N$. Hence the claim follows thanks to  the  Rellich-Kondrachov theorem;
	
	\item  if $p_{j}^-\leq N$, since $p^-\leq p_{j}^-\leq p_{j}^+$ and   the function $x\to \frac{x^2}{N+x}$ is increasing on $[0,+\infty)$, we get  that 
	$$\sigma<\frac{(p^-)^2}{N+p^-}\leq \frac{(p_{j}^+)^2}{N+p_{j}^+}.$$
		Thanks to the choice of $\sigma$ and  by applying  \eqref{sigmaest} with $r=r_j< \rho$, the above inequality implies $$p_{j}^->p_{j}^+-\sigma> p_{j}^+-\frac{(p_{j}^+)^2}{N+p_{j}^+}=\frac{N p_{j}^+}{N+ p_{j}^+} .$$
	Then, there exists  $\delta_j>0$  such that
	$$p_{j}^- -\delta_j> \frac{N p_{j}^+}{N+p_{j}^+} .$$ Hence, taking into account that $p_{j}^- -\delta_j< p_j^-\leq N$,  we obtain that  
	$$p_{j}^+<\frac {(p_{j}^- -\delta_j )N}{ N-p_{j}^- + \delta_j}= (p_{j}^- -\delta_j )^*$$ and, by  applying again  the  Rellich-Kondrachov theorem,  we get  that  $ W^{1,p_{j}^- -\delta_j}(\Om)$  is compactly  embedded
	into $L^{p_{j}^+}(\Om).$
	
\end{itemize}

Now,  for every $j\in F$, let $\varphi_j\in C^1(\Omega)$ be  such that $$ 0\leq \varphi_j\leq 1\,, \quad \varphi_j= 1 \hbox{ on } B_j\cap \Omega\,,\quad \varphi_j= 0\hbox{  on }\Omega\setminus \widetilde  B_j$$  and let $\sup_{j\in F}||\nabla \varphi_j||_{L^{\infty}(\Omega)}=C$.  % \color{red}( the constant depends on the radii, and due to the compactness can be chosen to be the same for every ball) \color{black}.
 Since   $p_{j}^--\delta_j<p(\cdot)$ on $\Omega\cap \widetilde B_j$,   using   the continuity of the  embedding $ L^{p(\cdot)}(\Omega\cap \widetilde B_j)   \hookrightarrow L^{p^-_{j}-\delta_j}(\Omega\cap \widetilde B_j)   $,    we get that there exists a  constant $C_1=C_1(\Omega)>0$ such that for all $ u\in W^{1,p(\cdot)}(\Om) $ it holds 
$$ ||u||_{L^{p(\cdot)}(\Omega\cap B_j)}  \leq  || \varphi_j u||_{L^{p(\cdot)}(\Omega\cap \tilde B_j)} \leq C_1  || \varphi_j u||_{L^{p^+_{j} }(\Omega\cap \tilde B_j)}  =  C_1|| \varphi _ju||_{L^{p^+_{j}}(\Omega)}.$$
%$$\leq  || \varphi_j u||_{W^{1,p_{0,j}^--\delta_j}(\Om)}\leq C_1 \| \varphi_j u| \|_{L^1(\Omega)}+ \|D(\varphi_j u) \|_{L^{p_{0,j}^--\delta_j}(\Om,\R^N)} $$

Hence, set $E_j:=\{v\in L^{p^+_{j}}(\Omega):\ v=  \varphi_j u, u\in W^{1,p(\cdot)}(\Om) \}$ endowed with the norm $\|\cdot \|_{L^{p^+_{j}}(\Omega)}$, for every $j\in F$ it holds  \beq\label{compE3} E_j   \ni \varphi_j u \longmapsto  u\in   L^{p(\cdot)}(\Omega)  \hbox{ holds and is continuous}.\eeq

%\beq\label{compE3} E_j=\{v\in L^{p^+_{j}}(\Omega):\ v=  \varphi_j u, u\in W^{1,p(\cdot)}(\Om) \}\hookrightarrow L^{p(\cdot)}(\Omega)  \hbox{ holds and  is continuous } \forall j\in F.   \eeq
Moreover,   for every $u\in W^{1,p(\cdot)}(\Om)$,
$$||\varphi_j u||_{L^{p^-_{j}-\delta_j}(\Omega)}=||\varphi_j u||_{L^{p^-_{j}-\delta_j}(\Omega\cap \tilde B_j)} \leq   C_1  ||\varphi_j u||_{L^{p(\cdot)}(\Omega\cap \tilde B_j)} \leq   C_1  || u||_{L^{p(\cdot)}(\Omega)} $$ and
\begin{align*}||D(\varphi_j u)||_{L^{p^-_{j}-\delta_j}(\Omega)}&=||D(\varphi_j u)||_{L^{p^-_{j}-\delta_j}(\Omega\cap \tilde B_j)}\\
 &\leq C_1 ( ||\varphi_j Du||_{L^{p(\cdot)}(\Omega\cap \tilde B_j)}+||u D\varphi_j ||_{L^{p(\cdot)}(\Omega\cap \tilde B_j)})\\
 &\leq   C_2   ||u||_{ W^{1,p(\cdot)}(\Om)} 
\end{align*}
i.e. \beq\label{compE2} W^{1,p(\cdot)}(\Om) \ni u\longmapsto \varphi_j u\in  W^{1,p^-_{j}(\cdot)-\delta_j}(\Om) \hbox{ holds and is continuous } \forall j\in F   .\eeq

By  combining  \eqref{compE2},  \eqref{compE1} and \eqref{compE3}, it easily follows that 
 the embedding $W^{1,p(\cdot)}(\Om) \hookrightarrow  L^{p(\cdot)}(\Omega)$  is compact.

 %$$\leq \underbrace{C_2}_{\color{red}W^{1,p_{0,j}^- -\delta_j}(\Om)\hookrightarrow L^{p_{0,j}^+}(\Om)} \sum_{j\in F} ||\varphi_j u||_{W^{1,p_{0,j}^- -\delta_j}(\Omega)} \leq \underbrace{C_3}_{\color{red} Lemma 2.6} \sum_{j\in F} ||\varphi_j u||_{L^1(\Omega)} + ||D(\varphi_j u)||_{L^{p_{0,j}^- -\delta_j}(\Omega)}  $$
%$$\leq  C_4 ||u||_{L^1(\Omega)} +  C_5 \sum_{j\in F}\left\{ || u||_{L^{p_{0,j}^- -\delta_j}(\tilde B_j)}+||D u||_{L^{p_{0,j}^- -\delta_j}(\Omega)} \right\}<+\infty  $$
%

%Therefore
%$$K_j  \subseteq  \{ v\in W^{1,p_{0,j}^--\delta_j}(\Om):v=0 \hbox{  in }  \Om \setminus \tilde B_j  \} \subseteq \{ v \in L^{ p^+_{0,j} }(\Om): v=0 \hbox{ a.e.   in }  \Om \setminus \tilde B_j  \} \subseteq L^{p_0(x)} (\Om)$$  (cf. also \cite[Corollary 3.3.4]{DHHR11}). 
%
%Hence for every $u\in W^{1,1}(\Omega)$ such that $ |\nabla u|\in  L^{p_{0}(x)}(\Omega)$ we have that    $\varphi_j u\in L^{p_0(x)}(\Omega)$ which implies $ u \in L^{p_0(x)}(\Omega\cap B_j)$ for every $j\in F$.  
%Therefore $ u\in L^{p_0(x)}(\Omega)$. 

Finally, it is easy to show that $W^{1,p(\cdot)}(\om;\R^d)$ is complete with respect to the norm $|||\cdot|||$. Since 
$|||\cdot|||\leq C ||\cdot||_{W^{1,p(\cdot)}(\Omega;\mathbb R^d)}$, by a classical result of Functional Analysis, we get that the two norms are equivalent on $W^{1,p(\cdot)}(\om;\R^d)$.
\qed

We conclude this subsection by recalling the following theorem which is of crucial importance for the results contained in Section \ref{secdimred}. It is a consequence of \cite[Theorem 11.2.7]{DHHR11} and properties of sets with Lipschitz boundary in \cite[Section 9.5]{DHHR11}.

\begin{thm}\label{0traceplog}
Let $\Omega$ be a bounded open set with Lipschitz boundary and let $p \in \mathcal P_b^{log}(
\Omega)$. If
$u \in W^{1,1}_0(\Omega)\cap W^{1, p(\cdot)}(\Omega)$ then
$ u \in W^{1,p(\cdot)}_0(\Omega)$.
\end{thm}

\color{black}

\subsection{$\Gamma$-convergence}
\noindent For an introduction to $\Gamma$-convergence, we refer to \cite{DM93}.
We recall the sequential characterization of  the $\Gamma$-limit when $X$ is a metric space.
%and when $X$is the dual of a separable Banach space.

\begin{prop}[\cite{DM93} Proposition 8.1] %\label{seqcharac}
Let $X$ be a metric  space and let $\f_{k}: X \ds \R \cup
\{\pm \infty\}$ for every $k\in \N$.
 Then  $\{\f_k\}$ $\Gamma$-converges to $\f$ with respect to the strong topology of  $X$ (and we write $\displaystyle  \Gamma(X)\hbox{-}\lim_{k\to +\infty}\f_{k}=\f$)  if and only if
\begin{description}
\item [(i)] {\rm($\Gamma$-liminf inequality)} for every $x\in X$ and for every sequence $\{x_{k}\}$ converging
to $x$, it is
$$  \f(x)\le \liminf_{k\to +\infty} \f_{k}(x_{k});$$
\item [(ii)]{\rm ($\Gamma$-limsup inequality)} for every $x\in X,$ there exists a sequence $\{x_{k}\}$  converging
to $x \in X$ such that
$$  \f(x)=\lim_{k\to +\infty} \f_{k}(x_{k}).$$
\end{description}
\end{prop}
We recall that the  $\Gamma\hbox{-}\lim_{k\to +\infty}\f_{k}$ is lower semicontinuous on $X$ (see \cite{DM93} Proposition 6.8).

%\begin{prop}\label{seqcharac'} Let $X$ be the dual of a separable Banach space and let $X$ be
%endowed with its weak* topology. Let $\f_n: X \ds \R \cup
%\{\pm \infty\}$ for every $n\in \N$. Assume that there exists $\Phi :X\to
%\R \cup \{\pm \infty\}$ such that:
%$$
%\lim_{\no x \no_X \to +\infty} \Phi(x)=+\infty,
%$$
% and  $\f_n\ge \Phi$ for every $n\in \N$.
%Then $(\f_n)$ $\Gamma$-converges to $\f$ with respect to the 
%weak* topology of $X$ (and we write $ \Gamma(w^*\hbox{-}X)\hbox{-}\lim_{n\to \infty}\f_n=\f$)   if and only if
%\begin{description}
%\item [(i)] for every $x\in X$ and for every sequence $(x_n)$ converging
%weakly$^*$ to $x\in X$, it is
%$$  \f(x)\le \liminf_{\ninf} \f_n(x_n);$$
%\item [(ii)] for every $x\in X$ there exists a  sequence $(x_n)$ converging
%weakly$^*$ to $x \in X$ such that
%$$  \f(x)=\lim_{\ninf} \f_n(x_n).$$
%\end{description}
%\end{prop}
%The proof of Proposition \ref{seqcharac'} easily follows the one of Proposition 8.10  in  \cite{DM93} with $X$ endowed with its weak* topology.
%Finally we recall also  that the function  $\f=\Gamma(w^*\hbox{-}X)\hbox{-}\lim_{n\to \infty}\f_n$ is weakly* lower semicontinuous  on $X$ (see \cite{DM93} Proposition 6.8) and when  $\f_n=\psi$  $\forall n\in\N$ then   $\f$   coincides with the  weakly*  lower semicontinuous (l.s.c.) envelope  of $\psi$, i.e. 
%\beq\label{rilassato}
%\f(x)=\sup\big\{h(x):\  \forall\, h:X\to\R\cup\{\pm\infty\} \   \ w^* \hbox
%{ l.s.c.}, \  h\le \psi\hbox { on } X\big\}
%\eeq
%(see Remark 4.5 in \cite{DM93}).

\begin{defn}\label{Gammafam}
We say that a family $\{\f_{\varepsilon}\}$ $\Gamma$-converges to $\f$, with respect to the topology considered on $X$ as $\varepsilon\to 0^+$,
if $\{\f_{{\varepsilon}_{k}}\}$ $\Gamma$-converges to $\f$ for all sequences $\{{\varepsilon}_{k}\}$ of positive numbers converging to $0$ as $k\to +\infty$.
\end{defn}
\smallskip
\color{black}

%For a comprehensive study of $\Gamma$-convergence we refer to the book of
%Dal Maso \cite{DM93} (for a simplified introduction see \cite{B}), while a
%detailed analysis of some of its applications to homogenization theory can
%be found in \cite{BDF}.

 \section{Some lower semicontinuity results}\label{SCIres}
In this section we provide some lower semicontinuity results for  different classes of integral functionals of the type
\begin{equation}\label{Ffunctional}
	F(u,p)=\int_{\Omega} \varphi(x, p(x), \nabla u(x)) \, dx
\end{equation} where the variable $p$ may stand for a variable growth exponent and $\Omega \subset \mathbb R^N$ is a open set. 
We discuss the two cases:   convexity and quasiconvexity of $\varphi(x,p,\cdot),$ in order to point out the different sets of assumptions.

Before proceeding, we recall the following definition.
\begin{defn} %A function $f:\Om\times \R^{k} \to\R$ is said to be a normal integrand if  
%\begin{itemize}
%\item [-] $f$ is  $\L^N\otimes \B_k$-measurable;
%\item  [-]  $f(x,\cdot)$ is lower semicontinuous  for $\mathcal L^2$-a.e. $x\in \Omega$;
%\end{itemize}
%\item   [-]   a Carath\`eodory integrand if
%\begin{itemize}
%\item  [-]  $f(\cdot, \xi)$ is measurable for every $\xi\in \bmicr \R^{kd} \emicr$;
%\item  [-]  $f(x,\cdot)$ is continuous   for a.e. $x\in \Omega$.
%\end{itemize}
A function  $\varphi:O\times\R^{m}\times \R^{s}\to (-\infty,\infty]$ is said to be a normal integrand if  
\begin{itemize}
\item  [-]  $\varphi$ is  $\L^N(O)\otimes\B (\mathbb{R}^{m}\times \mathbb{R}^s)$-measurable;
\item  [-] $\varphi(x,\cdot,\cdot)$ is lower semicontinuous  for ${\mathcal L}^N$- a.e. $x\in O$.
\end{itemize}
%\item [-] a Carath\'eodory integrand if
%\begin{itemize}
%\item  [-]  $f(\cdot, u, \xi)$ is measurable for every $(u,\xi)\in \R^d \times \bmicr \R^{kd} \emicr$;
%\item  [-]  $f(x,\cdot,\cdot)$ is continuous  for a.e. $x\in \Omega$.
%\end{itemize}
%A function  $f:\Om\times\R^d\times\R^{k}\to(-\infty,+\infty]$ is said to be a normal integrand if  
%\begin{itemize}
%\item  [-]  $f$ is  $\L^N\otimes\B_{d}\otimes \B_{k}$-measurable;
%\item  [-] $f(x,\cdot,\cdot)$ is lower semicontinuous  for a.e. $x\in \Omega$;
%\end{itemize}
Given a closed set $M\subseteq \R^m$,  a function  $\varphi:O\times M\times\R^s\to\R$ is said to be a Carath\'eodory integrand if 
\begin{itemize}
\item  [-]  $\varphi(\cdot, z,\xi)$ is  $\mathcal L^N(O)$-measurable for every $(z,\xi)\in M\times \R^m$;
\item  [-]  $\varphi(x,\cdot,\cdot)$ is continuous  for $\L^N$ a.e. $x\in O$.
\end{itemize}
\end{defn}

\subsection{The convex case}

If the function $\varphi$ in \eqref{Ffunctional} is convex in the last variable, then the functional $F$ is sequentially lower semicontinuous along sequences  $u_k\wto  u$ in $W^{1,1}(\Omega,\R^d)$ and  $p_k\to p_0$ in $L^1(\Omega;\mathbb R^m)$,  in view of the following result, shown by De Giorgi and Ioffe (see \cite[Theorem 5.8]{AFPbook}).
\begin{thm}\label{IoffeAFP}  Let $O$ be a bounded open subset of $\mathbb R^N$ and let
	$\varphi:O \times \mathbb R^{m}\times \R^s\to [0,+\infty]$ be a normal integrand such that   the map  $\xi \in \mathbb R^s \to \varphi(x,z,\xi)$ is convex for $\L^N$- a.e. $x \in O$ and every $z \in \mathbb R^m$.  Let  $\{z_k\}\subset  L^1(O,\R^m)$ and  $\{U_k\}\subset L^{1}(O; \mathbb{R}^{s})$. If  $z_k\to z_0$   in $L^1(O; \R^m)$ and $U_k \wto U_0$   in $L^1(O; \mathbb{R}^{s})$ then it holds	$$
	\int_O \varphi(x, z_0(x), U_0(x))dx \leq \liminf_{k \to +\infty} \int_O \varphi(x, z_k(x), U_k(x))dx.
	$$

	\end{thm}

As an application of Theorem \ref{IoffeAFP} with $m = 1$ and $s = d \times N$, we give the following result, that is crucial in the proof of Theorem \ref{semicon}.
\begin{cor}\label{Corvarphi} Let  $  \{p_k\}\subset  \mathcal P(\Omega) \cap L^1(\Omega)$ and  $\{v_k\} \subset W^{1,1}(\Omega; \mathbb{R}^d)$. If   $p_k\to p_0$  in $L^1(\Omega)$ and  $v_k \wto   v$ in $W^{1,1}_{loc}(\Omega; \mathbb{R}^d)$,  then 
  \beq\label{sci}
\int_{\Omega} |\nabla v(x)|^{p_0(x)} \, dx \le \, \liminf_{k \rightarrow + \infty} \int_{\Omega} |\nabla v_k(x)|^{p_k(x)}\, dx.
\eeq
In particular,  if $\{v_k\}$ satisfies the equiboundedness assumption \eqref{equilim}, then $\nabla  v \in L^{p_0(\cdot)}(\Omega;\mathbb R^{d \times N}).$ Finally, if  $\{v_k\}$  satisfies \eqref{equilim}, $\Omega$ is a Lipschitz bounded  open set, $p_0\in C(\overline{\Omega})$ and $v\in L^{1}(\Omega; \mathbb{R}^d)$, then $v\in W^{1,p_{0}(\cdot)}(\Om;\mathbb R^d)$.
\end{cor}
\proof 
In order to show \eqref{sci},  it is sufficient to apply Theorem \ref{IoffeAFP} on a   sequence of open sets $\{O_n\}$ such that $O_n\Subset \Omega$ and $\cup_n O_n=\Omega$, with   $\varphi: \Omega\times \R^{d\times N} \to [0,+\infty]$ given by
$$\varphi(x,z,U)=\left\{\begin{array}{ll}|U|^z &\hbox{ if } z\in [1,+\infty)\\
	 	+\infty &\hbox{ otherwise},
	 	\end{array}
 	\right.$$
and  $$z_k=p_k, \quad U_k=\nabla  u_k.$$ 
For every $n\in\N$, we obtain that
$$
\int_{O_n} |\nabla v(x)|^{p_0(x)} \, dx \le \, \liminf_{k \rightarrow + \infty} \int_{O_n} |\nabla v_k(x)|^{p_k(x)}\, dx\leq \liminf_{k \rightarrow + \infty} \int_{\Omega} |\nabla v_k(x)|^{p_k(x)}dx
$$
and  \eqref{sci} follows by sending $n\to + \infty.$ 
Finally, the additional assumption \eqref{equilim} combined with  inequality \eqref{sci}  implies that  $\nabla  v \in L^{p_0(\cdot)}(\Omega;\mathbb R^{d \times N}).$  The last part of the theorem follows by Proposition \ref{embp0x}.\qed

\color{black}
\begin{rem}\label{app} 
Let  $  \{p_k\}\subset   \mathcal P_b(\Omega) \cap L^1(\Omega)$  be such that, for some $k_0\in\N$,  it holds \beq\label{bounds}1<\inf_{k\geq k_0}  p_k^-.\eeq\begin{enumerate}
\item[\textnormal{(1)}]   If  $\{v_k\} \subset W^{1,1}(\Omega; \mathbb{R}^d)$  is equibounded in the sense of  \eqref{equilim} and $v_k \wto   v$ in $L^1(\Omega; \mathbb{R}^d)$ then $v_k \wto   v$ in $W^{1,1}(\Omega; \mathbb{R}^d)$.

Indeed, exploiting \eqref{bounds}, we can take   $1<r<\inf_{k\geq k_0}  p_k^-$. Hence, by \eqref{cpqomega}, it follows that
 \begin{eqnarray*}
\sup_{k\geq k_0} \|\nabla v_k  \|_{L^{r}(\Omega;\mathbb{R}^{d\times N})}  &\leq&  2( \mathcal{L}^N(\Omega)+1)   \sup_{k\geq k_0}    \|\nabla  v_k \|_{L^{p_k(\cdot)}(\Omega;\mathbb{R}^{d\times N})}  \\
&\leq& 2( \mathcal{L}^N(\Omega)+1) \sup_{k\geq k_0}    \left \{  \big( \rho_{p_k(\cdot)}(\nabla v_k)\big)^{\frac 1 {p_k^-} } ,  \big(\rho_{p_k(\cdot)}(\nabla v_k)\big)^{\frac 1 {p_k^+} } \right \}<+\infty.
\end{eqnarray*}

Since  $r>1$, we obtain  that  there exists a not relabeled subsequence $\{v_{k}\}$ such that $\nabla v_{k}\wto V$ in $L^r(\Omega;\mathbb R^{d\times N})$. 
This easily implies that $V=\nabla v$ and that $v_k\wto v$ in $W^{1,1}(\Omega; \mathbb{R}^d)$. This fact is used in the proof of Theorem \textnormal{\ref{semicon}}.

  \item[\textnormal{(2)}]  
 Let $\varphi: \Omega \times \R \times \mathbb{R}^{d \times N} \rightarrow \R$ be   a  $\mathcal{L}(\Omega)\otimes \mathcal{B}(\R\times \R^{d\times N})$-measurable function   satisfying the  growth condition 
\beq\label{crescita}
\varphi(x, p, \xi) \ge \, b(x) + C |\xi|^p\, \quad \hbox{ for } {\mathcal L}^N\hbox{-}\hbox{a.e. } x\in \Omega,\, \forall \xi\in \R^{d \times N}, \forall p\geq 1,
\eeq
where $C > 0$ and $b \in L^1(\Omega)$. 
Assume that   $p_k\to p_0$  in $L^1(\Omega)$ with $p_0 \in  \mathcal P(\Omega) \cap L^1(\Omega)$ and let  $\{v_k\} \subset W^{1,1}(\Omega; \mathbb{R}^d)$ be such that 
 \beq\label{limitatezza} \sup_k F(v_k,p_k)<+\infty,
 \eeq
where   $F$ is  the functional defined by  \eqref{Ffunctional}. If  $v_k \wto   v$ in $L^1(\Omega; \mathbb{R}^d)$
 then  $v_k \wto   v$ in $W^{1,1}(\Omega; \mathbb{R}^d)$ and $\nabla v\in L^{p_0{(\cdot)}}(\Omega;\mathbb R^{d \times N}).$ Indeed, \eqref{crescita} and \eqref{limitatezza} imply  \eqref{equilim} and then  it is sufficient to use  the  Part\textnormal{(1)} of this remark, combined with Corollary {\rm \ref{Corvarphi}.}
 
In particular, the coercivity assumption \eqref{crescita} is satisfied by integrands of the form
%integrand function of the functional \eqref{Ffunctional} has the form 
$\varphi(x,p, \xi)=W^p(x,\xi)$ with $p\geq 1$ and $$W(x, \xi)\geq b(x) + C |\xi|>0  \quad \hbox{ for  } \mathcal L^N \hbox{-}\hbox{a.e. } x\in \Omega,\, \forall \xi\in \R^{d \times N},$$
where $C > 0$ and $b(x) \in L^\infty(\Omega)$.
%{ \color{blue} Toglierei tutto il finale} It is sufficient to note that for every $p\geq 1$ it holds $$C^p|\xi|^p\leq (W(x, \xi)-b(x))^p\leq 2^{p-1}(W^p(x,\xi)+|b(x)|^p)$$
%which implies that $\varphi$ satisfies \eqref{crescita}.
\end{enumerate}
 \end{rem}

\subsection{The quasiconvex case}

The following special case  of \cite[Theorem 4.12]{MM}  holds (see also \cite[Theorem 7.4]{ARS}).
%(see Theorem \ref{MMteo4.12} in the Section  \ref{appendix})). 
It is important in order to show Theorems \ref{semicon} and \ref{semicon2}. 
For the reader's convenience,  after the statement, we provide an outline of the main steps necessary to deduce this  application from the results contained in \cite{MM}.

\begin{thm}\label{thmMM} 
	 Let $p_0 \in  C(\Omega)$ be such that  $p_0\in \mathcal P_b^{log}(A)$ for every $A\Subset \Omega$. 
Let $\varphi: \Omega \times \mathbb{R}^{d \times N} \rightarrow [0, + \infty)$ be a function such that
\begin{itemize}
 \item [-] $\varphi(\cdot,\xi)$ is $\mathcal L^N(\Omega)$-measurable for every $\xi \in \mathbb R^{d \times N}$; 
\item [-]  $\varphi(x, \cdot)$ is quasiconvex for $\mathcal{L}^N$-a.e. $x \in \Omega$;
\item [-] there exists $C > 1$ and $b(x) \in L^1_{\rm loc}(\Omega)$ with $b(x) \ge \, 0$ such that 
	\[
	0 \le \, \varphi(x, \xi) \le \, b(x) + C |\xi|^{p_0(x)}.
	\]
	\end{itemize}
	Then, for every sequence $\{v_k\} \subset W^{1,1}(\Omega; \mathbb{R}^d)$ with $v_k \rightarrow v$ in $L^1(\Omega; \mathbb{R}^d)$ and such that
	\begin{equation}
 \label{equilim-bis}
	\sup_k \int_{\Omega}  |\nabla v_k(x)|^{p_0(x)} \, dx < + \infty,
	\end{equation}
	we have that $v \in W^{1,{p_0(\cdot)}}_{loc}(\Omega; \mathbb{R}^d)$, $ |\nabla v|^{{p_0(\cdot)}} \in L^1(\Omega)$ and
	\begin{equation}\label{dislim}
	\int_{\Omega} \varphi(x,  \nabla v(x)) \, dx \le \, \liminf_{k \rightarrow + \infty} \int_{\Omega} \varphi(x,  \nabla v_k(x)) \, dx.
	\eeq
\end{thm}
\begin{proof}[Sketch of the proof]

We observe that,  by applying Remark \ref{app}(1) with $p_k\equiv p_0$, we have that the sequence $\{v_k\}$ weakly converges to $v$  in $W^{1,1}_{loc}(\Omega;\R^d)$. Hence,  by applying Corollary \ref{Corvarphi}, taking into account \eqref{equilim-bis}, we get that  $v \in W^{1,p_0(\cdot)}_{\rm loc}(\Omega; \mathbb{R}^d)$ and $ |\nabla v|^{p_0(\cdot)} \in L^1(\Omega)$.
\color{black}
 Hence, the lower semicontinuity result given by \eqref{dislim} follows as an application of \cite[Theorem 4.12]{MM}
recalling that, for every $A\Subset \Omega$,  
\begin{itemize}
\item   if $p_0\in C(A)$ and $p_0>1$ on $A$, then   the function $\psi(x,\xi)=|\xi|^{p_0(x)}$  enjoys the  {\sl Sobolev type property}  and  {\sl the  Rellich type property} on $A$ required by   \cite[Theorem 4.12]{MM} (see    \cite[Definition 3.2 and  Example 3.4]{MM}, and \cite[Lemmas 2.4 and 2.8]{CM}); 
\item   if  $p_0\in \mathcal P^{log}(A)$ then  $\psi$ verifies   {\sl the maximal property} on $A$,   see   \cite[Proposition 5.2 and Definition 4.4]{MM}.  Hence  $\psi$ satisfies   also the   {\sl density property} on $A$, see \cite[Definition  4.5 and  Proposition 4.6]{MM}.
 \end{itemize}

Hence, \cite[Theorem 4.12]{MM}  entails that 
 for every $A\Subset \Omega$
	\[
	\int_A \varphi(x,  \nabla v(x)) \, dx \le \, \liminf_{k \rightarrow + \infty} \int_{A} \varphi(x,  \nabla v_k(x)) \, dx\leq \liminf_{k \rightarrow + \infty} \int_{\Omega} \varphi(x,  \nabla v_k(x)) \, dx.
	\]
 
 By passing to the sup on $A\Subset \Omega$, we get the desired conclusion \eqref{dislim}.
\end{proof}

\medskip

%In the following theorems instead we investigate the lower semicontinuity of the integral functionals of kind \eqref{Ffunctional}
% \begin{equation}
% 	F(u,p)= \int_\Omega \varphi(x, p(x), \nabla u(x))dx,
% \end{equation}
%with respect to the $L^1(\Omega;\mathbb R^d)\times L^\infty(\Omega)$-strong convergence.
%\begin{rem}\color{red}  Note that in the  following  theorems  we can weaken the assumption on $\{p_k\}$ and require that $p_k\in L^\infty(\Omega)$  for every $k\in\N$.
%\end{rem}

\vspace{2mm}

 By using the arguments in \cite[Lemma 4.1]{ELN99} (see also \cite[Lemma 5.1]{LM}), we show the following key  result. 
\begin{lemma}\label{lemtec} Let $\Omega \subset \mathbb{R}^N$ an open set with finite Lebesgue measure. Let $W:\Omega\times \R^{d\times N}\to [0,+\infty)$ be a Borel function.  Let $a_0 \in L^\infty(\Omega)$ be such that  $a_0\geq 0$ $\mathcal{L}^N$-a.e. in $ \Omega$.
Let $p_0 \in \mathcal P_b(\Omega)$ and let  $\{p_k\}\subset L^\infty(\Omega)$ be  such that $ p_k\to p_0$ in $L^{\infty}(\Omega)$.
Then, for every $0<\delta<\min\{1/2, p_0^- -1\}$,   there exists $k_0\in\N$ such that, 
for every $k\geq k_0$ and for every $ U\in L^{1}(\Omega,\R^{d\times N})$, it holds
\beq \label{stimapx} \int_\Omega a_0(x)W^{p_0(x)- {\delta}}(x,U(x))dx\leq 2 {\delta} \mathcal{L}^N(\Omega)
+ (2 {\delta})^{-2 {\delta}}||a_0||_{L^\infty(\Omega)} \int_{\Omega} a_0(x) W^{p_k(x)}(x,U(x)) dx.\eeq
\end{lemma}

{\bf Proof.}  Let $0<\delta<\min\{1/2, p_0^- -1\}$. Thanks to the uniform convergence of $\{p_k\}$ to $p_0$, there exists $k_0\in\N$ and  a negligible measurable set $N\subseteq \Omega$ such that %\begin{equation}\label{vicinanza} \quad 
\[
1 <p_0(x)-{\delta}< p_k(x)<p_0(x)+{\delta}\quad \hbox{ for every }x\in\ \Omega\setminus N,
\]
%\end{equation}  
for every $k\geq k_0$. Hence, for every $x\in \Omega\setminus N$ and for every  $U\in L^{1}(\Omega,\R^{d\times N})$, if $2{\delta}<W(x,U(x))<1$,  then 
\beq\nonumber(2{\delta})^{2{\delta}}<(2 {\delta})^{p_k(x)-p_{0}(x)+{\delta}}< W^{p_k(x)-p_{0}(x)+{\delta}}(x,U(x))< 1\eeq
that implies
$$1<W^{p_{0}(x)-p_k(x)-{\delta}}(x,U(x))< (2{\delta})^{-2{\delta}},$$while, if $W(x,U(x))\leq 2\delta$, then $$W^{p_0(x)-\delta}(x,U(x))\leq (2\delta)^{p_0(x)-\delta}<2\delta,$$
being $2\delta<1$ and $p_0(x)-\delta>1.$
Hence, for every $k\geq k_0$ and for every  $U\in L^{1}(\Omega,\R^{d\times N})$,  it holds 
\begin{eqnarray*}
&&
\int_\Omega a_0(x) W^{p_0(x)-{\delta}}(x,U(x))dx\\
 &=& \int_{\{W(x,U(x))\leq 2{\delta} \}} a_0(x) W^{p_0(x)-{\delta}}(x,U(x))dx +\int_{\{2{\delta}<W(x,U(x))<1\}} a_0(x) W^{p_0(x)-{\delta}}(x,U(x))dx \\
&&+\int_{\{W(x,U(x))\geq 1\} }a_0(x) W^{p_0(x)-{\delta}}(x,U(x))dx\\
&\leq& 2\delta ||a_0||_{L^\infty(\Omega)}\mathcal{L}^N(\Omega)
+ (2{\delta})^{-2{\delta}}\bigg( \int_{\{ 2{\delta}<W(x,U(x))<1\}} W^{p_k(x)}(x,U(x)) dx+\int_{\{ W(x,U(x))\geq 1\} }W^{p_k(x)}(x,U(x))dx \bigg)\\
&\leq& 2\delta ||a_0||_{L^\infty(\Omega)} \mathcal{L}^N(\Omega)
+ (2{\delta})^{-2{\delta}} \int_{\Omega} W^{p_k(x)}(x,U(x)) dx.
\end{eqnarray*}
which gives the desired conclusion.
\qed

Now, we are in position to show the following general result which allows us to deduce  Theorems \ref{semicon} and \ref{semicon2} as immediate corollaries. 
\begin{thm}\label{thm3.7} Let $\Omega$, $W$, $p_0$,  $\{p_k \}$, $v$, $\{v_k\}$  be    as in Theorem {\rm \ref{semicon}} and  let $a_0$, $\{a_k\}$ be as in Theorem {\rm \ref{semicon2}}.  Then  $v$ satisfies \eqref{reglimv0} and 
 %\in W^{1,p_0(\cdot)}_{\rm loc}(\Omega; \mathbb{R}^d)$, $ |\nabla v|\in L^{p_0(\cdot)} (\Omega)$ }  
 $$\int_{\Omega} a_0(x) W^{p_0(x)}(x,\nabla v(x))dx\leq \liminf_{k \to +\infty} \int_{\Omega} a_k(x) W^{p_k(x)}(x,\nabla v_k(x))dx.$$

\end{thm}
 \vspace{2mm}
 
\textit{Proof.} {First of all, we note that, for every $A\Subset \Omega$,  since $\inf_{A}p_0>1$ and $p_k$ uniformly converges to $p_0$ as $k\to \infty$,   we have that there exists $k_0=k_0(A)\in\N$ such that 
$$\inf_{k\geq k_0} \infess_{A}p_k(x)>1.$$
In view of Remark \ref{app} (1), this implies  that $v_k\wto v$ in $W^{1,1}(A; \R^d)$  for every $A\Subset \Omega$. By applying  Corollary \ref{Corvarphi} in view of \eqref{equilim}, we get that  $v \in W^{1,p_0(\cdot)}_{\rm loc}(\Omega; \mathbb{R}^d)$ and $\nabla v\in L^{p_0(\cdot)}(\Omega)$. 
%{\color{blue}: quanto segue  mi sembra che sia vero in virtù della crescita dall'alto e dalla (1.19) e lo possiamo togliere Now, without loss of generality, we can suppose that $$\liminf_{k\to \infty} \int_{\Omega} W^{p_k(x)}(x,\nabla v_k(x))dx<+\infty.$$ }
Now we claim that for every $A\Subset \Omega$ it holds  \begin{equation}\label{claim} \int_{A}a_0(x)W^{p_0(x)}(x,\nabla v(x))dx\leq \liminf_{k\to +\infty} \int_{A} a_k(x)W^{p_k(x)}(x,\nabla v_k(x))dx.\end{equation}
Hence, by passing to the sup in \eqref{claim}  with respect to $A\Subset \Omega$, we get the desired conclusion. 

In order to show \eqref{claim}, for fixed $A\Subset \Omega$,  we note that 
\begin{eqnarray*}
&&\liminf_{k\to +\infty}  \int_{A} a_k(x)W^{p_k(x)}(x,\nabla v_k(x))  dx\\
&=& \liminf_{k\to +\infty}\left(\int_{A}(a_k(x)-a_0(x))W^{p_k(x)}(x,\nabla v_k(x))dx +  \int_{A}a_0(x) W^{p_k(x)}(x,\nabla v_k(x))dx\right)\\
&\geq &  \liminf_{k\to +\infty}\left( -||a_k-a_0||_{L^\infty(A)} \int_{A} W^{p_k(x)}(x,\nabla v_k(x))dx +  \int_{A}a_0(x) W^{p_k(x)}(x,\nabla v_k(x)) dx \right).
\end{eqnarray*}
On the other hand, \eqref{equilim} and \eqref{Wlingrowth} ensure $$ \liminf_{k\to + \infty}  \int_{A} W^{p_k(x)}(x,\nabla v_k(x)) dx <+\infty$$
and, since $a_k\to a_0$ in $L^\infty(\Omega)$, this  implies 
$$\lim_{k\to + \infty} ||a_k-a_0||_{L^\infty(A)} \int_{A} W^{p_k(x)}(x,\nabla v_k(x))dx=0.$$

Hence
\beq\label{anticipata}\liminf_{k\to +\infty}  \int_{A} a_k(x)W^{p_k(x)}(x,\nabla v_k(x))  dx\geq 
\liminf_{k\to +\infty}  \int_{A}a_0(x)W^{p_k(x)}(x, \nabla v_k(x))dx.
\eeq
 Since $\{p_k\}$ uniformly converges to $p_0$ on $A$ and $p_0>1$ on $A$, by  applying Lemma \ref{lemtec},  there exists  $k_0=k_0(A)\in\N$  such that, for every $k\geq k_0$, it holds  %\begin{equation}\label{vicinanza2}\quad 
   \beq \label{stimapxA} \int_A W^{p_0(x)- {\delta}}(x,U(x))dx\leq 2 {\delta} \mathcal{L}^N(A)
+ (2 {\delta})^{-2 {\delta}} \int_{A}  W^{p_k(x)}(x,U(x)) dx
\qquad \forall\, U\in L^{1}(A,\R^{d\times N}) .\eeq
  By applying \eqref{stimapxA} with $U=\nabla v_k$ and then   passing to the limit when $k\to \infty$, we obtain that
\beq\nonumber
\liminf_{k\to +\infty} \int_{A}a_0(x)W^{p_0(x)-\delta}(x,\nabla v_k(x))dx\leq  (2\delta)^{-2\delta}\liminf_{k\to +\infty}\int_{\Omega} a_0(x)W^{p_k(x)}(x,\nabla v_k(x)) dx + 2\delta ||a||_{L^\infty(A)}\mathcal L^N(A).
 \eeq
  Taking into account  the    growth condition \eqref{Wlingrowth}, we can apply Theorem \ref{thmMM} with the variable exponent $p_0(\cdot)-\delta\in P_b^{log}(A)$  to obtain  that 
\begin{align*}
	\int_{A}a_0(x)W^{p_0(x)-\delta}(x,\nabla v(x))dx\leq(2\delta)^{-2\delta}\liminf_{k\to +\infty}\int_{A} a_0(x)W^{p_k(x)}(x,\nabla v_k(x)) dx + 2\delta ||a_0||_{L^\infty(A)}\mathcal{L}^N(A).
\end{align*}

 By sending $\delta \rightarrow 0$ and by using Fatou's Lemma in the above inequality,  we finally  get the desired conclusion
\begin{align*}
\int_{A}a_0(x)W^{p_0(x)}(x,\nabla v(x))dx&\leq   \liminf_{k\to \infty}\int_{A} a_0(x)W^{p_k(x)}(x,\nabla v_k(x)) dx\\
&\leq \liminf_{k\to \infty}\int_{A} a_k(x)W^{p_k(x)}(x,\nabla v_k(x)) dx,
\end{align*}
where the last inequality follows by \eqref{anticipata}.
%Finally, if $a_k\to a$ uniformly and 
%$$\liminf_{k\to \infty}  \int_{\Omega} a_k(x)W(\nabla v_k(x))^{p_k(x)} dx <+\infty$$
 \qed

% we have that     $p_{\varepsilon}\in C (\overline{\Omega})$ and  
%\begin{align}\label{p-g1}
%	p_\varepsilon^-:=\infess_{\Omega}p_\varepsilon(x)\geq \infess_{\Omega}p(x) >1.
%\end{align}

\section{Dimensional reduction}\label{secdimred}
The aim of this section consists  proving the $\Gamma$-convergence results stated in Theorems \ref{dimredconv}, \ref{quasiconvex},  and \ref{ubqcxdimred} concerning   the dimension reduction problems defined in $\displaystyle \Omega_\varepsilon:=\omega\times (-\varepsilon/2, \varepsilon/2)$ where $\omega\subseteq \R^2$ is a bounded, Lipschitz open set.  To this end, in view of  Definition \ref{Gammafam}, in this section $\{\varepsilon_k\} \subseteq (0, 1/2)$ is any sequence  converging to  $0$.

We start proving a compactness result for energy bounded sequences, which, indeed, motivates the choice of the topology  in  our $\Gamma$-convergence results. Furthermore, according to the regularity of the variable exponent, modeling the point-dependent growth of the hyperelastic energy, we  obtain different results for the limiting deformation fields. Finally, the proofs of Theorems \ref{dimredconv}, \ref{quasiconvex} and \ref{ubqcxdimred} are provided showing that, for every infinitesimal sequence $\{\varepsilon_k\},$ the $\Gamma-$limit of $\{\mathcal F_{\varepsilon_k}\}$ exists and admits the same integral representation, when  $p\in \mathcal P_b^{log}(\Omega) $.
\color{black}

\subsection{Compactness of  bounded energy sequences}%\label{compactsec}

%\color{red} Observe that we have only used the continuity assumption on $p_0$ and not on $p_{\eps_k}$ to get the compactness below.
\color{black}  
 Let $\{\mathcal F_\varepsilon\}$ be the family of integral energies  in \eqref{Jeps}. The following compactness result for energy bounded sequences, relies on Poincar\'e inequality \eqref{Poincaretrace},  \color{black}   Corollary \ref{Corvarphi}  and Proposition \ref{embp0x}.

\begin{prop}\label{compactness} Let $f:\omega\times (-1/2,1/2)\times[1,\infty)\times \R^{3\times 3} \to \R$  be such that
\begin{itemize}
\item  [-]  $f$ is  $\L^2(\omega)\otimes\B ((-1/2,1/2)\times[1,\infty)\times \R^{3\times 3} )$-measurable;
\item  [-]  $f$  satisfies  \eqref{Vgrowthintro}.
\end{itemize}

Let $p \in \mathcal P_b(\Omega)$, let $\eps_k\to 0^+$ and let $\{p_{{\eps_k}}\}$ be  as in \eqref{pepsi}. Suppose that $\{u_{\varepsilon_k}\} \subset L^1(\Omega;\mathbb R^3)$ is a sequence  such that 
\beq\label{unifboud}
\sup_{k \in \mathbb{N}} \mathcal F_{\eps_k}(u_{\eps_k})<+\infty.
\eeq 
%where ${\mathcal F}_{\eps_k}$ is given by \eqref{Jeps}.
Then 
%and $u_\varepsilon \equiv 0$ on $\partial \omega \times(-1/2,1/2)$. $u_\varepsilon \equiv 0$ on $\partial \omega \times(-1/2,1/2)$, then 
\begin{enumerate}
	\item [\textnormal{(1)}] up to a not relabelled subsequence,  $\{u_{\eps_k}\}$  converges weakly in $
	W^{1,p^-}(\Omega;\mathbb R^3),$ (where $p^-$ is given by \eqref{p-p+}), to  a function  $u_{0}\in  W^{1,p^-}_{0, \partial_L \Omega} (\Omega;\mathbb R^3)$ (cf. Definition {\rm \ref{W1r-traccia}}), which, with an abuse of notation, we can identify with a function $u_0\in W^{1,p^-}_0(\omega;\mathbb R^3)$.
	\item [\textnormal{(2)}]   if  $p_0\in C(\overline  \omega)$ and $p_{\varepsilon_k} \to p_0 \hbox{ in } L_{loc}^1(\Omega)$, then     $u_0\in W^{1,p_0(\cdot)}(\Omega;\mathbb R^3)\cap  W^{1,p^-}_{0, \partial_L \Omega} (\Omega;\mathbb R^3)$ and it can be identified with $u_0 \in W^{1,p_0(\cdot)}(\omega;\mathbb R^3)\cap  W^{1,p^-}_{0}(\omega;\mathbb R^3)$;
	\item [\textnormal{(3)}]   if  $p_0\in \mathcal P_b^{log}(\omega)$ and   $p_{\varepsilon_k} \to p_0 \hbox{ in } L_{loc}^1(\Omega)$, then  $u_{0}$ can be identified with a function $u_0 \in W^{1,p_0(\cdot)}_0(\omega;\mathbb R^3)$;
	\item [\textnormal{(4)}] if  $p_0\in C(\overline  \omega)$ and $p_{\varepsilon_k} \to p_0 \hbox{ in } L^\infty(\Omega)$,  
  then  $u_{\varepsilon_k} \to u_0$ in $L^{p_{0}(\cdot)}(\Om;\mathbb R^3).$ 	
	\end{enumerate}
	
\end{prop}

%\textcolor{red}{(Michela): ho anticipato la remark perché nella dimostrazione si usava $p_0\in C(\overline  \Omega)$ mentre nell'enunciato si usava $\omega$.}

\begin{rem}	
Note that, if $p$ satisfies  the assumptions \eqref{p-p+} and \eqref{logp} on $\Omega$,  then $\{p_{{\varepsilon_k}}\}$ is  equicontinuous on $\overline{\Omega}$ and  equibounded (thanks to  \eqref{p-g2} and  \eqref{p-g1}). By applying Ascoli-Arzelà Theorem, it follows that  $\|p_{{{\varepsilon_k}}} - p_0\|_{L^{\infty}(\Omega)}\to 0$ with $p_0(x_{\alpha},x_3)=p(x_{\alpha},0).$

\end{rem}

\begin{proof}
\begin{enumerate}
\item  [(1)]  Thanks to the growth condition   \eqref{Vgrowthintro}, the assumption \eqref{unifboud} implies that   $u_{\eps_k}\in  W^{1,p^-}_{0, \partial_L \Omega} (\Omega;\mathbb R^3) \cap  W^{1,p_{{\varepsilon_k}}(\cdot)}( \Omega ;\mathbb R^3)  $ for every $k\in\N$   and that  there exists a  constant $C\geq 1$ such that
\begin{equation}\label{boundsulim}\displaystyle \sup_{k \in \mathbb{N}} \int_{\Om}| \nabla_\alpha u_{\eps_k} ,\tfrac{1}{\eps_k} \nabla_3 u_{\eps_k}|^{p_{\eps_k}(x)} dx \leq C.
\end{equation}

 By exploiting \eqref{p-g1},
 %we get that  
 %\leq  \sup_k \int_{\Omega} (| \nabla_\alpha u_{\eps_k}(x)|+1)^{p^-} \, dx\leq \sup_k \int_{\Omega} (| \nabla_\alpha u_{\eps_k}(x)|+1)^{{p_{\eps_k}(x)}} \leq C(p^+, |\Omega|)
 %<+\infty.$$
%$$\leq \sup_k \int_{\Omega} 2^ {p_{\eps_k}(x)-1}(| \nabla_\alpha u_{\eps_k}(x)|^{p_{\eps_k}(x)} +1) <2^ {p^+-1} \left(C+ |\Omega| \right) <+\infty.$$
%\begin{align} \sup_k \left\%{\int_{\Omega} | \nabla_\alpha u_{\eps_k}|^{p^-} \, dx,   \int_{\Omega} | \tfrac{1}{\eps_k}  \nabla_3 u_{\eps_k}|^{p^-} \, dx\right\}\leq C(p^+, |\Omega|)<+\infty.\label{r}
%\end{align}
%\begin{align}\label{r}
	% \|\nabla_3 u_{\eps_k}\|_{p^-}  \leq C(p^+, |\Omega|)  \eps_k \quad \forall k\in\N.
%\end{align}
both the sequences   $\{|\nabla u_{\eps_k}|\}$ and $\left\{\left|\frac{1}{\varepsilon_k} \nabla_3 u_{\varepsilon_k}\right|\right\}$ are bounded in $L^{p^-}(\Om)$ and, by using the Poincarè inequality  \eqref{Poincaretrace} with $\Gamma:= \partial\omega \times(-1/2,1/2)$, it turns out that $(u_{\eps_k})$ is bounded in $W^{1,p^-} (\Omega;\mathbb R^3)$, hence,
  $\{u_{\eps_k}\}$ admits a not relabelled subsequence which is weakly converging   in $W^{1,p^-} (\Omega;\mathbb R^3)$ to
$u_0 \in W^{1,p^-}_{0, \partial_L \Omega} (\Omega;\mathbb R^3)$, being the latter space  weakly closed.
%, i.e.%thus  there exists $u_0\in W^{1,p^-}_{0, \partial_L \Omega} (\Omega;\mathbb R^3)$  such that, up to a subsequence,  
%\begin{align}\label{weakp-}
%$u_{\varepsilon_k} \rightharpoonup u_0 \hbox{ in }W^{1,p^-}(\Omega; \mathbb R^3).$
%\end{align} 
On the other hand, %by \eqref{r} 
$\displaystyle \frac{\partial u_0}{\partial x_3}=0$ $\mathcal{L}^3$-a.e. in $\Omega$, hence, with an abuse of notation, we can identify $u_0$ with a function in $W^{1,p^-}_0(\omega;\mathbb R^3)$.

\item [(2)] The statement follows by Part (1) and Corollary \ref{Corvarphi}.

\item [(3)] In view of Theorem \ref{0traceplog}, under the additional assumption,  $p_0\in \mathcal P_b^{log}(\omega)$, the space $ W^{1,p^-}_0(\omega;\mathbb R^3)\cap W^{1,p_0(\cdot)}(\omega;\mathbb R^3)$   coincides with $W^{1,p_0(\cdot)}_0(\omega;\mathbb R^3)$.
\item  [(4)] Now, under the additional hypothesis 	that \beq\label{pepsiconv} p_{\varepsilon_k}\to p_0 \hbox{ in }L^\infty(\Omega), 
  \eeq
 we  show that  $u_{\varepsilon_k} \to u_0$ in $L^{p_{0}(\cdot)}(\Om;\mathbb R^3).$ \noindent Since $p_0\in C(\overline{\Omega})$ and $\partial \Omega$ is Lipschitz continuous, 
we can find a finite number of  open balls $(B_{i})_{i\in F}$ and  $(\tilde B_{i})_{i\in F}$,  such that $\bar B_{i}\ssubset \tilde B_i$ and \eqref{propball} hold.
%\begin{itemize}
	%\item $\bar \Om\subseteq \bigcup_{i\in F} B_{i}$;
	%\item $\Om\cap  \tilde B_{i}$ has %Lipschitz continuous boundary; 
	%\item 
  By following the same argument in Proposition \ref{embp0x},  
 for every $j\in F$ there exists $\delta_j>0$ such that     \eqref{compE1} holds with $p_{0,j}^+, p_{0,j}^-$  given by \eqref{p0j}.
%\end{itemize}

Let $\varphi_j\in C^1(\Omega)$   be such that $0\leq \varphi_j\leq 1$, $\varphi_j= 1$  on $B_j$, $\supp \varphi_j\subset \tilde B_j$  and $|\nabla \varphi_j|\leq C$.
Now we choose $0<\delta<\min_{j\in F} \delta_j$. Since $||p_{\eps_k}-p_0||_{L^\infty(\Omega)}\to 0$,  there exists  $k_0\in \mathbb N$ such that   $p_{\eps_k}(x)> p_{0}(x)-\delta$ for $k \ge k_0$ and for ${\mathcal L}^3$-a.e.  $x\in \Omega$. This, together with H\" older's inequality, implies that  for $k \ge k_0$  the embedding
\beq\label{compE4}  W^{1,p_{\eps_k}(\cdot)}(\Om;\mathbb R^3) \hookrightarrow W^{1,p_{0}(\cdot)-\delta_j}(\Om;\mathbb R^3) \hbox{  is continuous } \forall j\in F .\eeq

Moreover, by reasoning as in the proof of \eqref{compE2},  the embedding
\beq\label{compE5} W^{1,p_{0}(\cdot)-\delta_j}(\Om;\mathbb R^3) \ni u\mapsto \varphi_j u\in  W^{1,p^-_{0,j}(\cdot)-\delta_j}(\Om;\mathbb R^3) \hbox{ holds and is continuous } \forall j\in F   .\eeq
By  \eqref{compE4} , \eqref{compE5}  and \eqref{compE3}, it easily follows that, $\forall j\in F $ and for $k \ge k_0,$ the embedding
$$W^{1,p_{\eps_k}(\cdot)}(\Om;\mathbb R^3)  \ni u\hookrightarrow  \varphi_j u \in  L^{p_0(\cdot)}(\Omega;\mathbb R^3) \hbox{  is compact.}$$

Hence,  for every $j\in F$,
the sequence $\{u_{\eps_k} \}$ admits a subsequence $\{ u_{\eps^j_k} \}$ such that $\{\varphi_j u_{\eps^j_k}\}$ \color{black} converges strongly in   $L^{p_{0}(\cdot)}(\Om;\mathbb R^3)$ to a function $u_0^j$ for every $j\in F$.  Since $p^-\leq p_0(\cdot)$  in $\Omega$ then  $\varphi_j u_{\eps^j_k}\to u_0^j $ as $k \to + \infty$ in $L^{p^-}(\Om;\mathbb R^3)$. Then    $u_0^j=u_0$ for every $j\in F$ that implies $u_{\eps^j_k}  \to u_0$ in  $L^{p_{0}(\cdot)}(\Om\cap B_j;\mathbb R^3).$
Passing to subsequences repeatedly, we obtain that   $u_{\eps_k}  \to u_0$ in  $L^{p_{0}(\cdot)}(\Om;\mathbb R^3).$ 
%Finally, from \eqref{boundsulim},  by taking into account 
%the relationship \eqref{relaztotale} between modular and norm, we get  \begin{align}\label{estmod}
	%\|\nabla  u_{\eps_k}\|_{L^{p_{{{\varepsilon_k}}}(\cdot)}(\Omega;\mathbb R^{3 \times 3})} \leq \max\{C^{1/{p_{\eps_k}^-}}, C^{1/{p_{\eps_k}^+}}\}\leq C_1,
%\end{align}	
%where the latter constant does not %depend on $k$.
%Hence, by taking into account 
%the convergence \eqref{weakp-}  and by applying Corollary \ref{Corvarphi},  it follows that $|\nabla_\alpha u_0| \in L^{p_0(\cdot)}(\omega)$. Hence $u_{0}\in W^{1,p_0(\cdot)}(\omega;\mathbb R^3) $.
\end{enumerate}
\end{proof}
%Let  $\delta>0$ be such that $p_0^- -\delta > 1$ and let $\eps(\delta)>0$ be such that $||p_0-p_\varepsilon||_{L^{\infty}(\Omega)}<\delta$ for $0<\varepsilon < \varepsilon(\delta)$. Then, from \eqref{cpqomega} and \eqref{estmod}, we get that ,
%\begin{align}\label{4.4}
%	\|\nabla u_\varepsilon\|_{L^{p_0(\cdot)-\delta}}\leq 2(1+\mathcal L^N(\Omega)) \|\nabla u_\varepsilon\|_{L^{p_\varepsilon(\cdot)}}\leq C_2
%\end{align}
%
%Thanks to the reflexivity of the space $L^{p_0(\cdot)-\delta}(\Omega)$ we get that 
%\begin{align}\label{weakpo-delta}
%	u_\varepsilon \rightharpoonup u_0 \hbox{ in }W^{1,p_0(\cdot)-\delta}(\Omega),
%\end{align}
%Hence, from \eqref{4.4}, it holds
%\begin{align}
%	\label{semnorm}
%	\|\nabla_\alpha u_0\|_{L^{p_0(\cdot)-\delta}(\Omega)}\leq C_2.
%\end{align}
%that implies, by using again  the right hand side of \eqref{relaztotale},
%\begin{align}
%	\int_\Omega |\nabla_\alpha u_0|^{p_0(x)-\delta}dx \leq \max\{ C_2^{p_0^--\delta}, C_2^{p_0^+-\delta}\}.
%\end{align} 
%By Fatou's lemma, as $\delta \to 0$, we get that $|\nabla_\alpha u_0| \in L^{p_0(\cdot)}(\omega)$. %On the other hand, by \eqref{boundsulim}, we deduce that
%\begin{a

With the aim at proving Theorems \ref{dimredconv} \ref{quasiconvex} and \ref{ubqcxdimred}, we start with a preliminary lemma.

 \begin{lemma}\label{f_0}
Let $f:\omega\times(-1/2,1/2)\times[1,\infty)\times \R^{3\times 3} \to \R$  be such that 

\begin{itemize}
\item  [\textnormal{($1_f$)}]  $f$ is  $\L^2(\omega)\otimes\B ((-1/2,1/2)\times[1,\infty)\times \R^{3\times 3} )$-measurable.
\end{itemize}
Then
\begin{itemize}
	\item [\textnormal{($1_{f_0}$)}]
the function $\displaystyle f_0:\omega\times  (-1/2, 1/2)\times \R^{3\times 2} $, defined by 
\eqref{f0ours},  is  $\L^2(\omega)\otimes\B ((-1/2,1/2)\times[1,\infty)\times \R^{3\times 2} )$-measurable.
\end{itemize}

\noindent Furthermore, assuming that
\begin{itemize}
\item  [\textnormal{($2_f$)}] $f(x_{\alpha},\cdot,\cdot, \cdot)$ is lower semicontinuous  for ${\mathcal L}^2$- a.e. $x_{\alpha}\in \omega$;
 \item [\textnormal{($3_f$)}] there exists $C_1>0$ such that for $\mathcal L^2$-a.e. $x_{\alpha}\in \omega$ and for every  $\displaystyle (y,q,\xi)\in  (-1/2, 1/2)\times [1,\infty)\times \mathbb R^{3 \times 3}$,  it holds:
%\begin{equation}
%\label{coerc}
\[
	 f (x_{\alpha},y,  q,  \xi)\geq C_1|\xi|^q- \frac 1 {C_1} 
  \]
%\end{equation}
\end{itemize}
then, 
\begin{itemize}
\item  [\textnormal{($2_{f_0}$)}] $f_0(x_{\alpha},\cdot,\cdot,\cdot)$ is lower semicontinuous on $(-1/2,1/2)\times[1,\infty)\times \R^{3\times 2}$  for ${\mathcal L}^2$-a.e. $x_{\alpha}\in \omega$;
 \item [\textnormal{($3_{f_0}$)}] for every  $\displaystyle (y,q,\xi_\alpha)\in  (-1/2, 1/2)\times [1,\infty)\times \mathbb R^{3 \times 2}$ it holds

\begin{equation}\label{coercf_0}
	 f_0 (x_{\alpha},y,  q,  \xi_{\alpha})\geq C_1|\xi_{\alpha}|^q- \frac 1 {C_1}.
\end{equation}
\end{itemize}
Moreover, if \begin{itemize} 
	\item[\textnormal{$(4_f)$}] for $\mathcal L^2$-a.e.  $x_{\alpha}\in\omega$ the function  $f(x_{\alpha},y, q, \cdot)$ is convex  on $\R^3\times \R^3$ for every $(y,q)\in (-1/2, 1/2)\times [1,\infty)$,
	\end{itemize} then 
\begin{itemize}
	\item[\textnormal{($4_{f_0}$)}]
for $\mathcal L^2$-a.e.  $x_{\alpha}\in\omega$ the function  $f_0(x_{\alpha},y, q, \cdot)$ is convex on $\R^3\times \R^2$   for every $(y,q)\in (-1/2, 1/2)\times [1,\infty)$.
\end{itemize}

 Finally, if $f$
 satisfies \eqref{Vgrowthintro}, 
 then \begin{equation}\label{f01.24}
 C_1 |\xi_\alpha|^q- \frac{1}{C_1}\leq f_0(x_\alpha, y, q, \xi_\alpha)\leq C_2(|\xi_\alpha|^q +1),
 \end{equation}
 for $\mathcal L^2$-a.e. $x_\alpha \in \omega,$ and for every $(y,q,\xi_\alpha) \in (-1/2,1/2)\times [1,\infty)\times \mathbb R^{3 \times 2}$.

 \end{lemma}
\begin{proof} 

For simplicity, in the sequel we use  $f(x,y, q,\xi_{\alpha}, \xi_3)$ in place of $f(x,y, q,(\xi_{\alpha}, \xi_3))$. ($1_{f_0}$) follows by the fact that, by definition,  $f_0$ is the infimum of the family of functions $f(\cdot, \cdot,\cdot, \cdot, \xi_3)$ which are $\L^2(\omega)\otimes\B ((-1/2,1/2)\times[1,\infty)\times \R^{3\times 2} )$-measurable.
	
The proof of $(2_{f_0})-(3_{f_0})$ and \eqref{f01.24} follows along the lines the arguments of \cite[Proposition 1]{LDR}. The property $(4_{f_0})$ easily follows by the very definition \eqref{f0ours} of $f_0$.

  \end{proof}

 \subsection{The case of convex densities}
 \color{black}
 \begin{prop}\label{Gammaliminf}
Let $f:\omega\times (-1/2,1/2)\times[1,\infty)\times \R^{3\times 3} \to \R$   be a function satisfying $(1_{f})$, $(2_{f})$ and $(4_{f})$ of Lemma \textnormal{\ref{f_0}}.
Let $p \in \mathcal P_b(\Omega)$, let $\{p_{\eps_k}\}$ be  as in \eqref{pepsi} 
and let 
 ${\mathcal F}_{\varepsilon_k}$  
%${\mathcal J}$ 
be as in \eqref{Jeps}.
Let $\eps_k\to 0^+$ and assume that $p_{\varepsilon_k} \to p_0 \hbox{ in } L^1(\Omega)$
  with   $p_0=p_0(x_{\alpha})\in C(\overline{\Omega})$. Let  $\{u_{\eps_k}\}\subseteq 	{L^1(\Omega;\mathbb R^3)}$ %where the latter space is defined in \eqref{latW1px}, 
be a sequence satisfying \eqref{boundsulim} and such that $u_{\eps_k} \to u $ in $L^{1}(\Omega;\mathbb R^3).$ 
Then  
$${\mathcal F}_{p^-}(u)\leq \liminf_{k \to + \infty} {\mathcal F}_{\varepsilon_k}(u_{\eps_k}),
	$$
where  $\mathcal F_{p^-}:L^1(\Omega;\mathbb R^3)\to [0,+\infty]$ be the functional defined by \begin{equation}\label{Jp-def}
	{\mathcal 	F}_{p^-}(v):=\left\{
	\begin{array}{ll} \displaystyle \int_\omega {f}_0(x_\alpha, 0, p_0(x_{\alpha}), \nabla_\alpha v(x_\alpha))dx_\alpha , &\hbox{ if }v \in  W^{1,p_0(\cdot)}(\omega;\mathbb R^3)\cap W^{1,p^-}_0(\omega;\mathbb R^3), \\
		+\infty &\hbox{ otherwise in }L^{1}(\Omega;\mathbb R^3),
	\end{array}
	\right.
\end{equation} 
with   $f_0$ given by  \eqref{f0ours}.

\end{prop}
\color{black}

\begin{proof}
Without loss of generality, we can assume that $\liminf_{k\to \infty}{\mathcal F}_{\varepsilon_k}(u_{\eps_k}) <+\infty$,  and that, up to a not relabelled subsequence, such a  liminf is a limit. 
 Then, arguing as in Proposition  \ref{compactness} (see Part (1) and (3)) in view of \eqref{boundsulim},  $u_{\eps_k}$  
weakly converges to  $u$ in $W^{1,p^-}(\Omega;\mathbb R^3) $, and $u$ can be identified  with an element in $W^{1,p_0(\cdot)}(\omega;\mathbb R^3)\cap  W^{1,p^-}_0(\omega;\mathbb R^3)$. 
Moreover, %{\color{blue} since the sequence $\left\{\left|\frac{1}{\varepsilon_k} \nabla_3 u_{\varepsilon_k}\right|\right\}$ is  bounded in $L^{p^-}(\Om)$,} 
up to a subsequence,
\begin{align}\label{weakrescaled}\frac{1}{\varepsilon_{k_h}}\nabla_3 u_{\varepsilon_{k_h}} \rightharpoonup \bar{w} \in L^{p^{-}}(\Omega;\mathbb R^3),
	\end{align} where the latter function might depend on the selected subsequence.
 Now we  observe that the assumptions on $f$ guarantee that the  function $\varphi: \Omega\times \mathbb R^2 \times \R^{3\times 3} \to [0,+\infty]$, given by
	 $$\varphi(x,y,q, \xi)=\left\{\begin{array}{ll}f(x_{\alpha}, y,q,\xi) &\hbox{ if }(y,q)\in (-1/2,1/2)\times [1,+\infty)\\
	 	+\infty &\hbox{ otherwise},
	 	\end{array}
 	\right.$$
 	with $O= \Omega$, $m=2$ and $d= 3\times 3$,  satisfies all the assumptions  of Theorem \ref{IoffeAFP}. Since 
	$$z_k(x)= (\eps_k x_3, p_{\varepsilon_k}(x))\to (0,p_0(x_{\alpha}) )\hbox{ in } L^1(\Omega; \mathbb R^2)$$
and, thanks to \eqref{weakrescaled},   (up to a not relabeled subsequence), $$U_k(x)=\left(\nabla_{\alpha}u_{\eps_k}(x),\tfrac{1}{\varepsilon_{k}}\nabla_3 u_{\varepsilon_k}(x)\right) \rightharpoonup (\nabla u(x_\alpha), \bar w(x)) \hbox{ in } L^1(\Omega, \R^{3\times 3}), $$
 it follows that
\begin{eqnarray*} \lim_{\varepsilon_k \to 0^+}\int_{\Omega}f(x_\alpha,\varepsilon_k x_3, p_{\varepsilon_k}(x),\nabla_\alpha u_{\eps_k}(x),\tfrac{1}{\varepsilon_k} \nabla_3 u_{\eps_k}(x)) dx\nonumber 
	% &=& \lim_{\eps_{k_h}\to 0} \int_{\Omega}f(x_\alpha,\varepsilon_{k_h} x_3, p_{\varepsilon_{k_h}}(x),\nabla_\alpha u_{\eps_{k_h}}(x),\tfrac{1}{\varepsilon_{k_h}} \nabla_3 u_{\eps_{k_h}}(x)) dx \nonumber\\
& \geq&  \int_{\Omega}f(x_{\alpha},0, p_0(x_\alpha), \nabla_{\alpha}  u(x_\alpha), \bar{w}(x))dx  %\label{Joffeappl}
\\
	&\geq &\int_\Omega f_0(x_\alpha, 0, p_0(x_\alpha), \nabla_\alpha u(x_\alpha))dx_\alpha dx_3 \nonumber\\
	&=&\int_\omega f_0(x_\alpha, 0, p_0(x_\alpha), \nabla_\alpha u(x_\alpha))dx_\alpha, \nonumber
	\end{eqnarray*}

which concludes the proof, in view of the measurability property ($1_{f_0}$) satisfied by $f_0$ thanks to Lemma \ref{f_0}. 
\end{proof}

%\color{blue} Si può togliere
%\begin{rem}
	%\label{lscconJoffe2}
	%Proposition %\textnormal{\ref{Gammaliminf}} remains valid replacing the strong convergence in $L^1(\Omega)$ of $p_{\varepsilon_k}$ towards $p_0$ with the $L^1_{loc}(\Omega)$ convergence, provided that assumption $(2_f)$ is replaced by
%	$$
%	f(x_{\alpha},\cdot,\cdot, \cdot) \hbox{ is continuous  for a.e. }  x_{\alpha}\in \omega,
%	$$
%\begin{itemize}
%\item  [-]  $f$ is  $\L^2(\omega)\otimes\B ([-1/2,1/2]\times[1,\infty)%\times \R^{3\times 3} )$-measurable;
%\item  [{\textnormal{($5_f$)}}] %$f(x_{\alpha},\cdot,\cdot, \cdot)$ is %continuous  for a.e. $x_{\alpha}\in %\omega$.
%\end{itemize}
	%which enables us to invoke Theorem %\textnormal{\ref{DGI}} in place of %Theorem %\textnormal{\ref{IoffeAFP}} in %order to get the validity of %\eqref{Joffeappl}.
%\end{rem}

\begin{prop}\label{Gammlisup}
Let $\displaystyle f:\omega\times  (-1/2, 1/2)\times  [1,+\infty)\times \mathbb R^{3\times 3}\to [0,+\infty)$  be a function such that
\begin{itemize}
\item[-] $f(x_{\alpha},\cdot, \cdot, \cdot)$ is continuous for $\mathcal L^2$-a.e. $x_{\alpha}\in \omega$,
\item[-] $f(\cdot, y, q, \xi)$ is measurable for every $\displaystyle (y,q,\xi) \in  (-1/2, 1/2)\times [1,+\infty) \times 
\mathbb R^{3 \times 3}$. 
\end{itemize}
Morever assume that \eqref{Vgrowthintro} holds. 
Let $p \in \mathcal P_b(\Omega)$, let $\{p_{\eps_k}\}$ be  as in \eqref{pepsi}. 
Let $\eps_k\to 0^+$ and assume that $p_{\varepsilon_k} \to p_0 \hbox{ in } L^1(\Omega)$
  with  $p_0\in \mathcal{P}^{log}_b(\omega)$. 
  \color{black}
Then it holds

	\begin{equation}\label{Gammalimsup}
	\Gamma(L^{p_0(\cdot)})\hbox{-}\limsup_{k \to + \infty} {\mathcal F}_{{\varepsilon_k}} \le  {\mathcal F} 
	\end{equation}
	where $\mathcal F_{{{\varepsilon_k}}}$ and  ${\mathcal F}$ are defined, respectively,  by \eqref{Jeps} and \eqref{Jdef}.
		\end{prop}
		
\begin{proof}[Proof]
Define  
\begin{align}\label{Fdef} F:=\Gamma(L^{p_0(\cdot)})\hbox{-}{\limsup_{k \to + \infty}} \, \mathcal F_{{{\varepsilon_k}}}.
	\end{align}

%\Gamma(L^{p_0(x_\alpha)})-\limsup_{\varepsilon} {\mathcal J}_\varepsilon(u) \le \, \limsup_{\varepsilon} {\mathcal J}_{\varepsilon}(u_{\varepsilon}) = {\mathcal J}(u)
%\]
%under the extra assumption that $f_0$ in \eqref{f0ours} coincides with 
%\begin{equation}\label{f0ours2}f_0(x_\alpha,y, p, \xi_\alpha):= \inf_{\xi_3 \in \mathbb R^3}f(x_\alpha,y, p, \xi_\alpha,\xi_3) 
%\end{equation}
 By the very definition of  $\mathcal F$, we can restrict to prove that $F\leq \mathcal F$ on  $W_0^{1, p_0(\cdot)}(\omega;\mathbb R^3).$
First, we consider the case $u, w \in \mathcal{C}^{\infty}_0(\omega; \mathbb{R}^3)$.
%\st{By }\eqref{Vgrowthintro}, \st{it follows  that $\mathcal F(u)<+\infty$.} 
Set
$$u_{\varepsilon_k}(x):= u(x_{\alpha}) + \varepsilon_k \, x_3 \, w(x_{\alpha}),$$
we have that, as $k \to + \infty$, 
\[
u_{{{\varepsilon_k}}}(x_\alpha,x_3) \rightarrow u(x_{\alpha}) \qquad \textnormal{in $L^\infty(\Omega;\mathbb R^{3})$}
\]
and
\[
\left ( \nabla_{\alpha} u_{{\varepsilon_k}}(x), \tfrac{1}{\varepsilon_k} \nabla_3 u_{\varepsilon_k}(x)\right ) = \left (\nabla u(x_\alpha) + {{\varepsilon_k}} \, x_3 \nabla w(x_\alpha), w(x_\alpha) \right ) \rightarrow (\nabla u(x_\alpha), w(x_\alpha))  \hbox{ in }L^\infty(\Omega;\mathbb R^{3 \times 3}).
\]
Thus, in particular,  there exists $M>0$ such that $$\sup_{k \in \mathbb{N}} \|\left(\nabla u + {{\varepsilon_k}} \, x_3 \nabla w, w \right )\|_{L^\infty(\Omega;\mathbb R^{3\times 3})} \leq M.$$

Being $f(x_{\alpha}, \cdot,\cdot,\cdot,\cdot)$  continuous for $\mathcal L^2$-a.e.   $x_\alpha\in \omega$, by the above convergences   we get that, as $k \to + \infty$, 
\[
f \left (x_\alpha, {{\varepsilon_k}} x_3, {p_{{{\varepsilon_k}}}(x)},\nabla_{\alpha} u_{{{\varepsilon_k}}}(x), \tfrac{1}{\varepsilon_k} \nabla_3 u_{{{\varepsilon_k}}}(x) \right ) \rightarrow f (x_\alpha,0,p_0(x_{\alpha}),\nabla u(x_\alpha), w(x_\alpha))
\]

for $\mathcal L^2$-a.e.    $x_{\alpha}\in \omega$  and for every $x_3\in (- 1/2, 1/2 )$. Moreover, by \eqref{Vgrowthintro} and \eqref{p-g2}, it holds
\begin{align*}
\supess_{x\in \Omega}f\left (x_\alpha, \varepsilon x_3, p_{{{\varepsilon_k}}}(x), \nabla_{\alpha} u_{{{\varepsilon_k}}}(x), \tfrac{1}{{{\varepsilon_k}}} \nabla_3 u_{{{\varepsilon_k}}}(x) \right ) &\le \, C_2   \left(\supess_{x\in \Omega}\left|\nabla_{\alpha} u_{{{\varepsilon_k}}}, \tfrac{1}{{{\varepsilon_k}}} \nabla_3 u_{{{\varepsilon_k}}}\right|^{p_{\varepsilon_k}(x)}+1\right)\\
&=C_2\big(\supess_{x\in \Omega} \left|\left(\nabla u + {{\varepsilon_k}} \, x_3 \nabla w, w\right)\right| ^{p_{{{\varepsilon_k}}}(x)}+1\big)\\
&\leq C_2\big((1+M)^{ p^+}+1\big).
\end{align*}
Hence, by applying the Dominated Convergence Theorem,  
by \eqref{Vgrowthintro}
it follows that
\begin{align*}
F(u) \le \, \limsup_{k \to + \infty}  \mathcal F_{{{\varepsilon_k}}}(u_{\varepsilon_k}) &= \limsup_{k \to + \infty} \int_{\Omega} f \left (x_\alpha, {{\varepsilon_k}} x_3,{p_{{{\varepsilon_k}}}(x)}, \nabla_{\alpha} u_{{{\varepsilon_k}}}(x), \tfrac{1}{{{\varepsilon_k}}} \nabla_3 u_{{{\varepsilon_k}}}(x) \right )dx\\
&=\, \int_{\omega} f(x_\alpha, 0,{p_0(x_{\alpha})},\nabla u(x_\alpha), w(x_\alpha)) \, d x_{\alpha} < +\infty.\end{align*}

Being $w \in \mathcal{C}^{\infty}_0(\omega; \mathbb{R}^3)$ arbitrary, we get\[ F(u) \le \, \inf_{w \in \mathcal{C}^{\infty}_0(\omega; \mathbb{R}^3)} \int_{\omega} f (x_\alpha, 0,{p_0(x_{\alpha})},\nabla u(x_\alpha), w(x_\alpha)) \, d x_{\alpha}, \qquad \forall\, u\in \mathcal{C}^{\infty}_0(\Omega; \mathbb{R}^3).\]
Since
$ \mathcal{C}^{\infty}_0(\omega; \mathbb{R}^3)$ is dense in $L^{p_0(\cdot)}(\omega; \mathbb{R}^3)$, arguing in components and taking into account the growth condition \eqref{Vgrowthintro},  by  applying  Vitali-Lebesgue Dominated Convergence Theorem,  we get 
\begin{equation}\label{above}
F(u) \le \, \inf_{w \in L^{p_0(\cdot)} (\omega; \mathbb{R}^3)} \int_{\omega} f (x_\alpha, 0, p_0(x_{\alpha}),\nabla u(x_\alpha), w(x_\alpha)) \, d x_{\alpha}, \qquad \forall\, u\in \mathcal{C}^{\infty}_0(\Omega; \mathbb{R}^3).
\end{equation}

 Following the same argument in \cite[page 558]{LDR}, 
 %we can introduce the function $g: \omega \times \mathbb{R}^3 \rightarrow \mathbb{R}$ defined as %{\color{red} Forse si può minimizzare senza l'esponente proprio per quello che scrivete sotto}
%\[g(x_{\alpha}, \xi_3) = f(x_\alpha, 0, p_0(x_{\alpha}),\nabla u (x_{\alpha}), \xi_3)\]which is a Carath\'eodory function. Then, as in \cite[Section 1.4, page 235]{ET} 
there exists a measurable function $w_0=w_0(x_{\alpha})$  such that 
\[
f_0(x_\alpha,0, p_0(x_{\alpha}), \nabla u(x_{\alpha})) = \inf_{\xi_3 \in \mathbb R^3} [f(x_\alpha, 0, p_0(x_{\alpha}), \nabla u(x_{\alpha}), \xi_3)]=f(x_\alpha,0,p_0(x_{\alpha}),\nabla u(x_{\alpha}), w_0(x_{\alpha}))\]  for ${\mathcal L^2}$-a.e. $x_{\alpha}\in \omega.$
Thanks to \eqref{Vgrowthintro},  such a  function $w_0$ belongs to $L^{p_0(\cdot)}(\omega;\mathbb R^3)$.
%Observe that the infimum does not depend on the exponent $p_0(x_{\alpha})$, being $x_{\alpha}$ fixed.
%At this point 
%\[
%\inf_z [W(\nabla_{\alpha} u(x_{\alpha}), z)]^{p_0(x_{\alpha})} = W_0(\nabla_{\alpha} u(x_{\alpha}))^{p_0(x_{\alpha})}
%\]
Hence, by using \eqref{above} and  by recalling the definition of $\mathcal F$,  we get that for every $u \in C^{\infty}_0(\omega; \mathbb{R}^3)$ it holds 
\begin{equation}\label{eq:dentroC}
F(u)   \leq \int_{\omega} f(x_\alpha,0,p_0(x_{\alpha}), \nabla_{\alpha} u(x_{\alpha}), w_0(x_{\alpha}))  dx_{\alpha}= \, \int_{\omega} f_0(x_\alpha,0, p_0(x_{\alpha}), \nabla u(x_{\alpha})) \, d x_{\alpha}=\mathcal F(u).
\end{equation}
%\textcolor{magenta}{questo ultimo passaggio non dovrebbe funzionare a meno che la $f$ non dipende da $x_3$ oppure a meno che non ci siano delle ipotesi che ci garantiscono che il minimo non dipende da $x_3$} A me sembra funzionare 
In view of the coercivity assumption on $f$ (see \eqref{Vgrowthintro}), we can apply  Lemma \ref{f_0} to obtain that  the function $f_0(x_{\alpha},\cdot,\cdot,\cdot)$ is lower semicontinuous on $(-1/2,1/2)\times[1,\infty)\times \R^{3\times 2}$  for ${\mathcal L^2}$-a.e. $x_{\alpha}\in \omega$; being also the supremum of continuous function, we obtain that $f_0(x_{\alpha},\cdot,\cdot,\cdot)$ is continuous on $(-1/2,1/2)\times[1,\infty)\times \R^{3\times 2}$  for ${\mathcal L^2}$- a.e. $x_{\alpha}\in \omega$. Moreover, taking into account \eqref{f01.24}, $$f_0(x_\alpha,0, p_0(x_{\alpha}), \xi_{\alpha}) \leq C_2(|\xi_{\alpha}|^{p_0(x_{\alpha})}+1), \qquad \hbox{for ${\mathcal L^2}$-a.e. }x_\alpha \in \omega, \forall\, \xi_\alpha \in \mathbb R^{3\times 2}.$$
Therefore $\mathcal F$  is strongly continuous in $W^{1, p_0(\cdot)}(\omega;\mathbb R^3)$.  Moreover,  by the very definition (see \eqref{Fdef}),  the functional $F$ is lower semicontinuous with respect to the strong convergence in $L^{p_0(\cdot)}(\Omega;\mathbb R^3) $. Hence, by applying  the density of $C_0^{\infty}(\omega;\mathbb R^3)$ in $W^{1, p_0(\cdot)}_0(\omega;\mathbb R^{3}) $ (see Remark \ref{remdens}, part (1)), \eqref{eq:dentroC} implies that
\color{black}
\[
F(u) \le \, {\mathcal F}(u) \qquad \forall u \in W_0^{1, p_0(\cdot)}(\omega;\mathbb R^3).
\]
By the very definitions of $F$ and ${\mathcal F}$, this entails \eqref{Gammalimsup}.
\end{proof}

\begin{proof}[Proof of Theorem {\rm \ref{dimredconv}}]
	By Theorem \ref{0traceplog}, \color{black} when $p_0\in \mathcal P^{log}_b(\omega)$, the functionals $\mathcal F$ in \eqref{Jdef} and $\mathcal F_{p^-}$ in \eqref{Jp-def} coincide. Hence the result follows by applying  Propositions \ref{Gammaliminf} and \ref{Gammlisup},  noticing  that  for every $\{u_{\eps_k}\}\subseteq 	{L^1(\Omega;\mathbb R^3)}$ satisfying $\displaystyle\liminf_{k \to + \infty}{\mathcal F}_{\varepsilon_k}(u_{\eps_k}) <+\infty$, the growth  condition in
\eqref{Vgrowthintro},  implies \eqref{boundsulim}.  
		\end{proof}

\begin{cor}
	\label{dimredconvcor}Under the same assumptions and with the  same notation of Theorem {\rm \ref{dimredconv}} it results that
	$$\Gamma(L^{p_0(\cdot)})\hbox{-}\lim_{k \to + \infty}\mathcal F_{{{\varepsilon_k}}}= \mathcal F.
	$$
	
\end{cor}
\begin{proof}[Proof]
The result is a direct consequence of Theorem \ref{dimredconv} together with Proposition \ref{Gammalimsup}, since the first ensures that $\mathcal F$ is a lower bound for the liminf of ${\mathcal F_{{{\varepsilon_k}}}}(u_{{{\varepsilon_k}}})$ for every $u_{{{\varepsilon_k}}} \to u$ in $L^{p_0(\cdot)}(\Omega;\mathbb R^3)$ while the upper bound is achieved using the recovery sequence constructed in Proposition \ref{Gammalimsup} which is converging in $L^{p_0(\cdot)}(\Omega; \R^3)$. 
\end{proof}

\begin{rem}
	%\label{remdimredconv}
 \begin{itemize}
		\item[{\rm (i)}]  A careful inspection of the result of Corollary \textnormal{\ref{dimredconvcor}} shows that it holds also if we replace,  in the definition of the functional $\mathcal{F}_{\varepsilon_k},$ the null boundary datum on $\partial_L\Omega$ by any other function $u_0\in \cap_{{{\varepsilon_k}} >0} W^{1,p_{{{\varepsilon_k}}}(\cdot)}(\Omega;\mathbb R^3)\cap W^{1,p_0(\cdot)}(\Omega;\mathbb R^3)$ satisfying  $u_0 \equiv u_0(x_\alpha).$
		\item[{\rm (ii)}]  The $\Gamma$-convergence result with respect to the $L^{p_0(\cdot)}$-strong convergence, stated in Corollary \textnormal{\ref{dimredconvcor}}, appears as a natural convergence under the assumption \eqref{pepsiconv} on $\{p_{{\varepsilon_k}}\}$.
	\end{itemize}
\end{rem}

\subsection{The nonconvex  case}

This subsection is devoted to the proofs of Theorems \ref{quasiconvex} and \ref{ubqcxdimred}. 
 With this aim, we recall a relaxation result which extends, in the case of variable exponents, to integral functionals with trace constraints on the boundary the more general results proven in \cite{MM} in the uncostrained case. \color{black}
 For the readers' convenience we state it in  our framework, which deals with functionals with growth of the type $|\xi|^{p(x)}$, referring to \cite[Theorem 5.1]{Sichev2011} for the result stated in full generality.

\begin{thm}\label{Sychev5.1}
Let $U$ be a bounded open subset of $\mathbb R^N$  with  Lipschitz boundary, let  $p \in \mathcal P_b(U)$  and consider a Carath\'eodory integrand  $\varphi: U \times \mathbb{R}^d \times \mathbb{R}^{d \times N} \to  \mathbb R$ satisfying  the growth condition
	\begin{equation}\label{GC}
	C_1 + C_2 |\xi|^{p(x)}\leq \varphi(x,  \xi) \leq C_3 + C_4 |\xi|^{p(x)},
	\end{equation}
	for a.e. $x \in U$, every $u \in \mathbb R^d$ and $\xi \in \mathbb R^{d\times N}$,
	with $0<C_2 <C_4$.
	 Let $ J:L^{1}(U;\mathbb R^d)\to [0,+\infty] $ be the functional defined  by $$
 J(u):=\left\{
\begin{array}{ll} \displaystyle \int_\Omega \varphi(x,\nabla u(x)) dx  &\hbox{ if } \varphi(\cdot, \nabla u(\cdot))\in L^1(U), \\
\\
	+\infty &\hbox{ otherwise in }L^{1}(U;\mathbb R^d).
	\end{array}
\right.
$$
 Assume that  $u_0 \in W^{1,p^-}(U;\mathbb R^d)$   is such that $ J(u_0) < \infty$. If there exists a sequence $u_k \in W^{1, p^-}(U;\mathbb R^d)\cap C^1(U_k;\mathbb R^d)$, where  $$U_k : = \{x \in U: {\rm dist}(x, \partial U) > 1/k \}, \quad  k \in \mathbb N,$$ such that  $u_k \lfloor_{\partial U}= u_0\lfloor_{\partial U}$,   
	 $u_k\to u_0$  in $W^{1, p^-}(U;\mathbb R^d)$ and $J(u_k)\to J(u_0)$, 	as $k \to +\infty$, then there exists a sequence $\tilde u_k \in W^{1,p^-}(\Omega;\mathbb R^d) \cap W^{1,\infty}(U_k;\mathbb R^d)$ such that   $\tilde u_k \lfloor_{\partial U}= u_0\lfloor_{\partial U}$, $\tilde u_k\wto u_0$ in $W^{1, p^-}(U;\mathbb R^{d})$ 
 and   $$\lim_{k\to \infty}J(\tilde u_k)= \int_U Q\varphi(x, \nabla u_0(x))dx,$$
	where $Q\varphi$ denotes the quasiconvex envelope of $\varphi(x,\cdot)$.
	\end{thm}

\color{black}

\noindent Now, exploiting the same techniques of Theorem \ref{semicon}, we are in position to show  the lower bound inequality.

 \begin{prop} \label{lbqcxdimred} 
%Let  $p\in \mathcal P_b(\Omega)$ %satisfying
 %\eqref{p-p+} 
 %and let $p_{\varepsilon_k}$ be as in \eqref{pepsi}. Assume that there exists $p_0\in {\mathcal P}_b^{log}(\omega)$,  such that $p_{\varepsilon_k}\to p_0$ in $L^\infty(\Omega)$ as $k\to + \infty$. {\textcolor{red}{controllare ovunque se va bene $k \rightarrow + \infty$ o $\varepsilon \rightarrow 0^+$}. Let $W:\omega\times \R^{3\times 3}\to [0,+\infty)$ be a Carath\'eodory  function satisfying the assumptions of Theorem \textnormal{\ref{quasiconvex}}.} 
 Let $\{\mathcal I_{\varepsilon_k}\}$ and $\mathcal I$ be the functionals defined by \eqref{Jepsf} and \eqref{Jqc}, respectively. 
  %and we  $W(x, \xi) \equivW(x_{\alpha},\xi)$, 
  %and let $\{u_{{\eps_k}}\}\subseteq W^{1,p_{\varepsilon_k}(x)}_L(\Omega;\mathbb R^3)$ be
%such that $u_{\eps_k} \to u $ in $L^{1}(\Omega;\mathbb R^3)$. 
%Then  $$\mathcal I(u)\leq \liminf_{k\to +\infty} \mathcal I_{\varepsilon_k}(u_{\eps_k}).$$
Under the assumptions of Theorem {\rm \ref{quasiconvex}},   
$$\mathcal I(u)\leq \liminf_{k \to + \infty} \mathcal I_{\varepsilon_k}(u_{\eps_k}),$$
for every $L^1(\Omega;\mathbb R^3) \ni u_{\eps_k} \to u $ in $L^{1}(\Omega;\mathbb R^3)$ as ${k \to + \infty}.$
\color{black}
\end{prop}
\begin{proof} 
% {\color{red} Note that the function $\tilde W(x,\xi)=W_0(x_{\alpha}, \xi_{\alpha})$ is an upper semicontinuous function  on $\Omega\times \R^{3\times 3}$ such that
%\begin{equation}\label{W0cresc}
%0  \le \tilde W(x,\xi)  \le \,  C_2( |\xi|+1).
%\end{equation}
%Moroever fwe  define
%\begin{equation}\label{Qg}
%Q\tilde W(x,\cdot):=\sup\{h:\R^{3\times 3}\to\R\, \colon   h \ \hbox{quasiconvex}\ \hbox{and } h\le \tilde W(x,\cdot)\}
%\end{equation}
%$$=\sup\{h:\R^{3\times 3}\to\R\, \colon   h \ \hbox{quasiconvex}\ \hbox{and } h\le  W_0(x_\alpha,\cdot) \}:=QW_{0}(x_\alpha,\cdot).$$
%
%Therefore, thanks to \cite[Theorem 4.4.1]{Bu},  there exists a Carath\'eodory  integrand $\tilde f$   such that $\tilde f(x,\cdot)$ is quasiconvex for a.e $x\in \Omega$  and a negligible set $N\subseteq \Omega$ 
%\begin{align}\label{QW0qo}QW_0(x_{\alpha},\xi)=\tilde f(x,\xi) \hbox{ for  a.e.  }(x,\xi)\in ( \Omega\setminus N) \times \R^{3\times 3} .
%\end{align} }In particular as observed in \cite{Bu}, \begin{align}\label{due}
% \tilde f(x,\xi) \leq \tilde W(x,\xi)=W_0 (x_\alpha,\xi)
% \end{align}
% a.e. in $\Omega$.

\color{black}
Without loss of generality, assume that $\displaystyle \liminf_{k \to + \infty}\mathcal I_{\varepsilon_k}(u_{\eps_k}) <+\infty$.
By exploiting the  assumption $p_0 \in \mathcal P_{b}^{log}(\omega)$, Proposition \ref{compactness} %\textcolor{magenta}{\st{(parts (3)-(4)),}}  
ensures that  $u$ can be identified with a  function in $ W^{1,p_0(\cdot)}_0(\omega;\mathbb R^3)$   and   $u_{\eps_k}\to u$ in  $ L^{p_0(\cdot)}(\Omega;\mathbb R^3)$ and $u\wto u$ in $W^{1,p^-}(\Omega;\mathbb R^3).$  %If  $p_0^-\geq N$ then,  thanks to  the Sobolev compact  embedding of $W^{1,p_0^-}(\Om)$ in $C(\bar \Om)$, for $0<\eps<\eps_0$ we have the Sobolev compact  embedding of $W^{1,p_{\eps}(x)}(\Om)$   in $L^{p_{0}(x)}(\Om).$ %{\color{red}Cosa succede se $p_0^-< N\leq p_0^+$?} %\noindent Assume  $p_0^+<N.$ 
%From \eqref{boundsulim} and Corollary \ref{Corvarphi}   (inserire la coercività) we get that $$\int_{\Om}| \nabla_\alpha u(x_{\alpha}, y)|^{p_0(x_{\alpha})}dx_{\alpha} dy \leq \liminf_{\eps\to 0}\int_{\Om}| \nabla_\alpha u_{\eps}|^{p(x_{\alpha},\eps x_3)}<C$$
%that implies  $$\int_{\omega}| \nabla_\alpha v(x_{\alpha})|^{p_0(x_{\alpha})}dx_{\alpha} <C$$
%	that is $v\in W^{1,p_0(\cdot)}(\omega)$.
	%	Note that, taking into account \eqref{Wcresc},  $W_0$ satisfies all the assumptions of Proposition 9.5 in \cite{Dac}. Hence   $QW_0$ is a Carath\'eodory integrand and $QW_0(x, \cdot)^{p_0(x_{\alpha})}$ is quasiconvex for a.e. $x\in \Omega$,   
Taking into account that $p_0^->1$,  let  $0<\delta<\min\{1/2, p_{0}^--1\}$ and let $k_0\in \N$ be such that 
$$1<p_0(x)-\delta< p_{\eps_k}(x)<p_0(x)+\delta \quad \hbox{ for $\mathcal{L}^3$ \hbox{-} a.e. }x\in \Omega$$  for every $k\geq k_0$.
By using  \eqref{stimapx} we have that
$$2\delta |\Omega|
+ (2\delta)^{-2\delta}  \int_{\Omega} W^{p_{\eps_k}(x)}(x_\alpha, \nabla_\alpha u_{{\eps_k}}(x),\tfrac{1}{\varepsilon_k} \nabla_3 u_{\eps_k}(x)) dx \geq \int_{\Omega}W^{p_0(x_{\alpha})-\delta}(x_\alpha, \nabla_\alpha u_{{\eps_k}}(x),\tfrac{1}{{\eps_k}} \nabla_3 u_{\eps_k}(x))dx.$$

Hence, by passing to the liminf as ${k \to + \infty}$ and  by applying  the very definition of $W_0$ and  $Q(W_0^{p_0(\cdot)})$, we get  that
	\begin{align}\label{key*}2\delta |\Omega|
+ (2\delta)^{-2\delta} \liminf_{k \to + \infty}  \int_{\Omega} W^{p_{\eps_k}(x)}(x_\alpha, \nabla_\alpha u_{\eps_k}(x),\tfrac{1}{\varepsilon_k} \nabla_3 u_{\eps_k}(x)) dx  \geq \\
\liminf_{k \to + \infty} \int_{\Omega}W^{p_0(x_{\alpha})-\delta}(x_\alpha, \nabla_\alpha u_{\eps_k}(x),\tfrac{1}{\varepsilon_k} \nabla_3 u_{\eps_k}(x))dx \geq \nonumber\\\liminf_{k \to + \infty} \int_{\Omega}W_0^{p_0(x_{\alpha})-\delta}(x_{\alpha}, \nabla_{\alpha}  u_{\eps_k}(x))dx \geq\nonumber\\
\liminf_{k \to + \infty}  \int_{\Omega}Q(W^{p_0(x_{\alpha})-\delta}_0)(x_{\alpha},\nabla_{\alpha}  u_{\eps_k}(x)) dx. \nonumber
\end{align}
Then, by  Theorem \ref{thmMM},  we have \begin{align*}\liminf_{k \to + \infty}  \int_{\Omega}Q(W^{p_0(x_{\alpha})-\delta}_0)(x_{\alpha},\nabla_{\alpha}  u_{\eps_k}(x))dx&\geq \int_{\Omega}Q(W^{p_0(x_{\alpha})-\delta}_0)(x_{\alpha},\nabla_{\alpha}  u(x_{\alpha}))dx_{\alpha} dy\\
=  \int_{\omega}Q(W^{p_0(x_{\alpha})-\delta}_0)(x_{\alpha},\nabla_{\alpha}  u(x_{\alpha})) dx_{\alpha}.
\end{align*}
Hence, by combining \eqref{key*} and the above inequality, as $\delta\to 0^+$, we obtain 
$$\liminf_{k \to + \infty} \int_{\Omega} W^{p_{\eps_k}(x)}(x_\alpha, \nabla_\alpha u_{\eps_k}(x),\tfrac{1}{\varepsilon_k} \nabla_3 u_{\eps_k}(x)) dx\geq \liminf_{\delta\to 0^+}\int_{\omega}Q(W^{p_0(x_{\alpha})-\delta}_0)(x_\alpha,\nabla_{\alpha}  u(x_{\alpha}))dx_{\alpha}.$$ 
Now we claim that, \begin{equation}\label{contW0}
|W_0(x_{\alpha},\xi_\alpha)-W_0(x_{\alpha},\eta_\alpha)|\leq w(x_{\alpha}, |\xi_\alpha-\eta_\alpha|)\hbox{ for every } \xi_\alpha, \eta_\alpha\in \R^{3\times 2}, \hbox{ for } \mathcal L^2 \hbox{-a.e. } x_{\alpha}\in \omega.\end{equation}
Indeed, let $N\subseteq \omega$ be  such that \eqref{Wcresc}, \eqref{growthWxab1} and \eqref{unifcont} hold for every
$x_{\alpha}\in \omega\setminus N$ and for every $\xi,\eta \in \R^{3\times 3}.$
Hence, thanks to the coercivity assumption \eqref{Wcresc}, for a.e. $x_{\alpha}\in \omega\setminus N$ and for every $\xi_{\alpha}\in \R^{3\times 2}$ there exists $\bar \xi_{3}\in \R^3$ such that 
$W_0(x_{\alpha},\xi_{\alpha})=W(x_{\alpha},\xi_{\alpha},\bar \xi_3) .$
Therefore, for every $\eta_\alpha\in \R^{3\times 2}$,
\vspace{2mm}
\begin{align*}W_0(x_{\alpha},\eta_{\alpha})\leq W(x_{\alpha},\eta_{\alpha},\bar \xi_3) \leq w(x_{\alpha}, |(\xi_\alpha,\bar \xi_3)-(\eta_\alpha,\bar \xi_3|) +W(x_{\alpha},\xi_{\alpha},\bar \xi_3) =\\
w(x_{\alpha}, |\xi_\alpha-\eta_\alpha,|) +W(x_{\alpha},\xi_{\alpha},\bar \xi_3)\leq w(x_{\alpha}, |\xi_\alpha-\eta_\alpha|)+W_0(x_{\alpha},\xi_{\alpha}) .
\end{align*}
By changing the role of $\xi_{\alpha}$ and $\eta_{\alpha}$ we get \eqref{contW0}, which implies that the map $(x_{\alpha}, q, \xi)\mapsto W_0^{q}(x_{\alpha},\xi)$ is a Carath\'eodory integrand. In particular, this in turn yields that  for every $q\geq 1$ and $\mathcal L^2$-a.e. $x_\alpha \in \omega$ the function $W_0^{q}(x_{\alpha},\cdot)$ is Borel measurable and locally bounded; moreover, it is bounded from below by the right-hand side of \eqref{coercf_0}, hence, its quasiconvex envelope $Q(W_0^{q})(x_{\alpha},\cdot)$ satisfies the assumptions of \cite[Theorem 6.9]{Dac} and can be represented as
$$
Q(W_0^{q})(x_{\alpha},\xi_\alpha) =\inf\left\{\frac{1}{\mathcal L^2(D)}\int_D W_0^{q}(x_{\alpha},\xi_\alpha + \nabla_\alpha \psi(y_\alpha))d y_\alpha : \psi \in W^{1,\infty}_0(D;\mathbb R^3)  \right\}$$
with $D$ any bounded open subset of $\mathbb R^2$.

Hence, by using \cite[part 3(b) in Proposition 9.5]{Dac},  the map $(x_{\alpha}, q, \xi)\mapsto Q(W_0^{q})(x_{\alpha},\xi)$ is a Carath\'eodory integrand too. Therefore, for $\mathcal L ^2$-a.e. $x_{\alpha}\in \omega$ and for every $\xi_{\alpha}\in \R^{3\times 2}$ it holds \begin{equation}\label{quasienvconv}\lim_{\delta \to 0^+}Q(W^{p_0(x_{\alpha})-\delta}_0)(x_\alpha,\xi_{\alpha})= Q(W^{p_0(x_{\alpha})}_0)(x_\alpha,\xi_{\alpha})\end{equation}
and,  by using Fatou's Lemma, we finally get 
$$\liminf_{\varepsilon_k \to 0^+} \int_{\Omega} W^{p_{\eps_k}(x)}(x_\alpha, \nabla_\alpha u_{\eps_k}(x),\tfrac{1}{\varepsilon_k} \nabla_3 u_{\eps_k}(x)) dx\geq \int_{\omega}Q(W^{p_0(x_{\alpha})}_0)(x_{\alpha},\nabla_{\alpha}  u(x_\alpha))dx_{\alpha}.$$ \end{proof}
Now we are in position to prove the upper bound inequality for Theorem \ref{quasiconvex}.

 \begin{prop} \label{gammalimsupquasic}
%Let  $p\in \mathcal P_b(\Omega)$ 
%satisfying
 %\eqref{p-p+} 
% and define $p_{\varepsilon_k}$ as in \eqref{pepsi} with  $\eps_k\to 0$.   Let $W:\omega\times \R^{3\times 3}\to [0,+\infty)$ be a Carath\'eodory  function satisfying the assumptions of Theorem \textnormal{\ref{quasiconvex}}. 
 %Assume that  $p_{\varepsilon_k}\to p_0$ in $L^\infty(\Omega)$ with  $p_0\in {\mathcal P}_b^{log}(\omega)$.
 Under the assumptions of Theorem {\rm \ref{quasiconvex}}, it results that
$$\Gamma(L^1)\hbox{-}\limsup_{k \to + \infty} \mathcal I_{{{\varepsilon_k}}}\leq \mathcal I,$$
where $\mathcal I_\varepsilon$ and $\mathcal I$ are defined by \eqref{Jepsf} and \eqref{Jqc}, respectively.
\end{prop}

\begin{proof}
First of all, we observe that the functional  $ {\mathcal I}_{\varepsilon_k}$  coincides with the functional ${\mathcal F}_{\varepsilon_k}$ in \eqref{Jeps} when  $f:\omega\times (-1/2,1/2)\times[1,\infty)\times \R^{3\times 3} \to \R$ is given by $f(x_{\alpha}, y, q, \xi)=W^q(x_{\alpha}, \xi)$.
Moreover, the assumptions on $W$ ensure that for $\mathcal L ^2$-a.e. $x_{\alpha}\in \omega$ and $\xi_\alpha\in \R^{3\times 2}$, it holds  \beq\label{W0p}\inf_{\xi_3 \in \mathbb R^3}(W^{p_0(x_{\alpha})}(x_\alpha, \xi_\alpha,\xi_3))=
%\min_{\xi_3 \in \mathbb R^3}(W^{^{p_0(x_{\alpha})}}(x_\alpha, \xi_\alpha,\xi_3))=\left(\min_{\xi_3 \in \mathbb R^3}(W(x_\alpha, \xi_\alpha,\xi_3)\right)^{p_0(x_{\alpha})}
W_0^{p_0(x_{\alpha})}(x_{\alpha}, \xi_\alpha)
\eeq
and the infimum above is attained.

\color{black}
Hence,  by  applying Proposition \ref{Gammlisup},  we have that
\begin{equation}\label{I1}I:=\Gamma(L^{p_0(\cdot)})\hbox{-}\limsup_{k \to + \infty} {\mathcal I}_{\varepsilon_k} \le  {\mathcal I}_1
\end{equation}
where $$
 {\mathcal 	I}_1(u):=\left\{
 	\begin{array}{ll} \displaystyle \int_\omega {W}^{p_0(x_{\alpha})}_0(x_\alpha, \nabla_\alpha u(x_\alpha))dx_\alpha , &\hbox{ if }u \in  W^{1,p_0(\cdot)}_0(\omega;\mathbb R^3), \\
	\\
 		+\infty &\hbox{ otherwise in }L^{1}(\Omega;\mathbb R^3).
 	\end{array}
 	\right.
$$

We introduce  $$
J(u):=\left\{
 	\begin{array}{ll} \displaystyle \int_\omega {W}^{p_0(x_{\alpha})}_0(x_\alpha, \nabla_\alpha u(x_\alpha))dx_\alpha , &\hbox{ if }{W}^{p_0(\cdot)}_0(\cdot, \nabla_\alpha u(\cdot))\in L^1(\omega) \\
	\\
 		+\infty &\hbox{ otherwise in }L^{1}(\Omega;\mathbb R^3).
 	\end{array}
 	\right.
$$
Thanks to   \eqref{growthWxab1} and by applying  Lemma \ref{f_0},  $W_0: \omega\times \R^{3\times 2} \to \R$ is   a Carath\'eodory integrand 
satisfying 
\begin{equation}\label{W0cresc}
 C_1 |\xi_{\alpha}|-\frac 1 {C_1} \le \,W_0(x_\alpha, \xi_{\alpha})  \le \,  C_2( |\xi_{\alpha}|+1)  \qquad \hbox{for $\mathcal{L}^2$-a.e. } x_{\alpha}\in \omega, \hbox{ for every }\xi\in \R^{3\times 3}.
\end{equation}
This, clearly, implies that also the map %\varphi: \omega \times \mathbb R^{3 \times 2} \to \mathbb R$, defined as $\varphi(x_\alpha,\xi_\alpha)
$(x_\alpha,\xi_\alpha) \to W_0^{p_0(x_\alpha)}(x_\alpha,\xi_\alpha)$,  is a Carath\'eodory integrand satisfying the growth condition \eqref{GC}. 

In order to show that $I\leq \mathcal I $,  without loss of generality, we can  consider the case when $\mathcal I(u)<+\infty$. Then  $ u \in W_0^{1, p_0(\cdot)}(\omega;\mathbb R^3)$ and, thanks to \eqref{GC}, $ J(u)<+\infty.$ 
Moreover,  being $p_0\in \mathcal P^{\log}_b(\omega)$, %with %$p^+<\infty$, 
\color{black}
Part (1) in Remark \ref{remdens}   entails that $C_0^\infty(\omega;\mathbb R^{3})$ is dense in $W^{1,p_0(\cdot)}_0(\omega;\mathbb R^3)$. Since  $J$  is continuous with respect to the convergence in $W^{1,p_0(\cdot)}_0(\omega;\mathbb R^3)$, we get that  $J$ satisfies the assumptions of  Theorem \ref{Sychev5.1} on $U=\omega$. Hence,    we can infer the existence of a sequence $\tilde u_{{\varepsilon_k}} \in W_0^{1,p_0^-}(\omega;\mathbb R^{3}) \cap W^{1,\infty}(\omega_{{\varepsilon_k}};\mathbb R^{3})$, such that   $\tilde u_{{\varepsilon_k}}\wto u$ in $W^{1, p_0^-}(\omega;\mathbb R^3)$ 
 and  \begin{equation}\label{recoverySychev2}
 \lim_{k \to +\infty} \int_\omega W_0^{p_0(x_\alpha)}(x_\alpha ,\nabla \tilde u_{{\varepsilon_k}}(x_\alpha)) d x_\alpha=  \int_\omega Q (W_0^{p_0(x_\alpha)})(x_\alpha, \nabla u(x_\alpha))d x_\alpha= \mathcal I(u)<+\infty.
\end{equation}
 In particular, thanks to the coercivity of $W_0^{p_0(\cdot)}$, \eqref{recoverySychev2} implies that 
  there exists $k_0\in\N$ such that  $\nabla \tilde u_{{\varepsilon_k}}\in L^{ p_{0}(\cdot)}(\omega;\mathbb R^{3\times 2})$ for every $k\geq k_0$. By exploiting Proposition \ref{embp0x},  this implies that  $u_{{\varepsilon_k}}\in W^{ 1,p_{0}(\cdot)}(\omega;\mathbb R^{3})$. Finally, by   Theorem \ref{0traceplog}, we obtain that $u_{{\varepsilon_k}} \in W^{1,p_0(\cdot)}_{0}(\omega;\mathbb R^3)$ for $k\geq k_0$. Hence $\mathcal I_1(\tilde u_{{\varepsilon_k}}) = J(\tilde u_{{\varepsilon_k}}) $ for every  $k\geq k_0$.   Moreover, by using \eqref{W0cresc} and the Poincar\'e inequality in $W^{1,p_0(\cdot)}_{0}(\omega;\mathbb R^3)$ (see \eqref{poinc}),  it is easy to obtain that   $\tilde u_{{\varepsilon_k}}\wto u$ in $W^{1, p_0(\cdot)}(\omega;\mathbb R^3)$. Therefore, by Proposition \ref{embp0x}, $\tilde u_{{\varepsilon_k}}\to u$ in $L^{ p_0(\cdot)}(\omega;\mathbb R^3)$  and then $\tilde u_{{\varepsilon_k}}\to u$ in $L^{ p_0(\cdot)}(\Omega;\mathbb R^3)$. 
 Since $I$ is lower semicontinuous with respect to such a convergence, by \eqref{I1} and \eqref{recoverySychev2},  we get that 
$$I(u)\leq \liminf_{k \to + \infty} I(\tilde u_{{\varepsilon_k}}) \leq  \lim_{k \to + \infty} \mathcal I_1(\tilde u_{{\varepsilon_k}}) = \lim_{k \to +\infty}  J(\tilde u_{{\varepsilon_k}}) =\int_\omega Q (W_0^{p_0(x_\alpha)})(x_\alpha, \nabla \tilde u(x_\alpha))d x_\alpha= \mathcal I(u).
$$

%{\color{blue} che è diverso dal dire che we get that 
 %$$\mathcal I(u)\geq \limsup_{k\to \infty} \mathcal I_{\varepsilon_k}(\tilde{u}_{\eps_{k}})$$}
\end{proof}

\begin{proof}[Proof of Theorem \textnormal{\ref{quasiconvex}}]
It follows from Propositions \ref{lbqcxdimred} and \ref{gammalimsupquasic}.
\end{proof}

Finally we prove Theorem \ref{ubqcxdimred}, splitting the proof in a  double inequality. We underline that the adopted arguments are very similar to those exploited in the proof of Theorem \ref{quasiconvex}, relying for the lower bound on techniques stemmed from Theorem \ref{semicon2}. Indeed we provide for the reader's convenience only the proof of the main differences. Arguing as in Theorem \ref{thm3.7}, it is possible to show that the result generalizes to the case when $W$ depends also on $x_{\alpha}$; however we prefer to focus just on the model behaviour of an energy density of product-type between the spatial variable and a gradient dependence.

\begin{prop}\label{semicon2dimred} 
  Let $\{\mathcal J_{\varepsilon_k}\}$ and $\mathcal J$ be the functionals defined by in \eqref{Jepsfa(x)} and \eqref{Jqca(x)}, respectively. 
  %and we  $W(x, \xi) \equivW(x_{\alpha},\xi)$, 
  %and let $\{u_{{\eps_k}}\}\subseteq W^{1,p_{\varepsilon_k}(x)}_L(\Omega;\mathbb R^3)$ be
%such that $u_{\eps_k} \to u $ in $L^{1}(\Omega;\mathbb R^3)$. 
%Then  $$\mathcal I(u)\leq \liminf_{k\to +\infty} \mathcal I_{\varepsilon_k}(u_{\eps_k}).$$
Under the assumptions of Theorem {\rm \ref{ubqcxdimred}},   
$$\mathcal J(u)\leq \liminf_{k \to + \infty} \mathcal J_{\varepsilon_k}(u_{\eps_k}),$$
for every $L^1(\Omega;\mathbb R^3) \ni u_{\eps_k} \to u $ in $L^{1}(\Omega;\mathbb R^3)$ as $\varepsilon_k \to 0^+.$
\color{black}

%$$ \liminf_{\eps\to 0} \int_{\Omega}a(x_{\alpha}, \eps x_3) W(\nabla_\alpha u_{\eps},\frac{1}{\varepsilon} \nabla_3 u_{\eps})^{p_{\eps}(x)}  dx \geq %\liminf_{\eps\to 0} \int_{\Omega}a(x_{\alpha}, \eps x_3) QW_0(\nabla_\alpha u_{\eps})^{p_{\eps}(x)}   dx_{\alpha} dy $$
% \int_{\omega}a(x_{\alpha}, 0)  Q (W_0^{p_{0}(x_\alpha)}) (\nabla_\alpha u(x_\alpha))   dx_{\alpha}dy,$$
	%where $W_0$ 
	%is given by 
	%\begin{equation}\label{W0}
		%W_0(\xi_\alpha):= \inf_{\xi_3 \in \mathbb R^3}W( \xi_\alpha,\xi_3), \quad \qquad  \forall \xi_{\alpha}\in \R^{3\times 2}
	%\end{equation}
	%and  \begin{equation}\label{QW02} Q W_0:=Q(W_0)=....
	%\end{equation}
\end{prop}

{\bf Proof.} Without loss of generality, we can assume that $\displaystyle \liminf_{k \to + \infty}\mathcal J_{\varepsilon_k}(u_{\varepsilon_k})<+\infty$. Thus, by \eqref{a-}$$  \liminf_{k \to + \infty} \int_{\Omega} W^{p_{{{\varepsilon_k}}}(x)}(\nabla_\alpha u_{{{\varepsilon_k}}}(x),\tfrac{1}{{{\varepsilon_k}}} \nabla_3 u_{\varepsilon_k}(x)) dx \leq \frac{1}{a^-} 
\liminf_{{k \to +\infty}}\mathcal J_{{\varepsilon_k}}(u_{{\varepsilon_k}})<+\infty.
$$
In turn, thanks to \eqref{coerca(x)}, %this implies
%\beq\label{equilimdimred}
	%\sup_{\varepsilon} \int_{\Omega} %\left|\nabla_\alpha u_{\varepsilon_k}(x),\tfrac{1}{\varepsilon_k} \nabla_3 u_{\varepsilon_k}(x)\right|^{p_{\varepsilon_k}(x)} \, dx < + \infty,
	%\eeq
\eqref{boundsulim} holds. \color{black}
Hence, in view of our assumptions on $p_{\varepsilon_k}$ and $p_0$, by (1), (3) and (4) in Proposition \ref{compactness}, we have that, up to a not relabelled subsequence $u_{\varepsilon_k}\rightharpoonup u $ in $W^{1, p^-}(\Omega;\mathbb R^3)$ and strongly in $L^{p_0(\cdot)}(\Omega;\mathbb R^3)$ with $u \in W^{1,p_0(\cdot)}_0(\omega;\mathbb R^3)$. 

Furthermore, $a_{\varepsilon_k}$, $p_{\varepsilon_k}$and $u_{\varepsilon_k}$ satisfy the assumptions of Theorem \ref{semicon2}, hence we can mimick the arguments therein. %Observe that 
%\begin{align}\label{MMappl}
%	\int_{\Omega}a(x)W(\nabla v(x))^{p_0(x)}dx\leq   \
%\end{align}
%Finally, if $a_k\to a$ uniformly and 
%$$\liminf_{k\to \infty}  \int_{\Omega} a_k(x)W(\nabla v_k(x))^{p_k(x)} dx <+\infty$$
%Note that \eqref{equilim} and \eqref{Wlingrowth} 
%which implies {\color{blue}
%	$$\lim_{k\to \infty} ||a_k-a||_{\infty} \int_{\Omega} %(W(\nabla v_k(x)))^{p_k(x)}dx=0.$$} 
 More precisely, by reasoning as in the first part of Theorem \ref{semicon2} (see \eqref{anticipata}) we  have that 
\beq\nonumber\liminf_{k \to + \infty}  \int_{\Omega} a_{{{\varepsilon_k}}}(x)W^{p_{{{\varepsilon_k}}}(x)}(\nabla_\alpha u_{{{\varepsilon_k}}}(x), \tfrac{1}{{{\varepsilon_k}}}\nabla_3 u_{{{\varepsilon_k}}}(x))  dx
	\geq  \liminf_{k \to + \infty}\int_{\Omega} a_0(x_\alpha)W_0^{p_{{{\varepsilon_k}}}(x)}(\nabla_\alpha u_{{{\varepsilon_k}}}(x)) dx,
\eeq
where $W_0$ is the function appearing in \eqref{QW0nox}.
Moreover, by applying the same argument of Lemma \ref{lemtec} and Theorem  \ref{thmMM},  for every  $0<\delta<\min\{1/2, p_0^- -1\}$  it holds 
\begin{align}
& (2\delta)^{-2\delta}\liminf_{k \to + \infty}\int_{\Omega} a_0(x_\alpha)W_0^{p_{\varepsilon_k}(x)}(\nabla_\alpha u_{{{\varepsilon_k}}}(x)) dx + 2\delta ||a||_{L^\infty(\Omega)}{\mathcal L}^3(\Omega) \nonumber\\
&\geq  \liminf_{k \to + \infty} \int_{\Omega}a_0(x_\alpha)W_0^{p_0(x_\alpha)-\delta}(\nabla_\alpha u_{\varepsilon_k}(x))dx \nonumber\\
 &\geq \liminf_{k \to + \infty} \int_{\Omega}a_0(x_\alpha)Q(W_0^{p_0(x)-\delta})(\nabla_\alpha u_{\varepsilon_k}(x)) dx \label{eq1} \\
&\geq \int_\Omega a(x_\alpha)Q(W_0^{p_0(x_\alpha)-\delta})(\nabla_\alpha u(x_\alpha))dx. \nonumber 
\end{align}
Now, arguing as in  Proposition \ref{lbqcxdimred} (see \eqref{quasienvconv}) i.e. exploiting the fact that  quasiconvex envelope $Q(W_0^{q})$ satisfies the assumptions of \cite[Theorem 6.9]{Dac} , 
it follows that 
\beq\label{eq2}\lim_{\delta \to 0}Q(W_0^{p_0(x_{\alpha})-\delta})(\xi_{\alpha})= Q(W_0^{p_0(x_{\alpha})})(\xi_{\alpha})\eeq
 Hence, combining \eqref{eq1} and \eqref{eq2} and sending $\delta \to 0$, Fatou's lemma  gives 
 \begin{align*}
	\liminf_{k \to + \infty}\int_{\Omega}a_{{{\varepsilon_k}}}(x) W^{p_{{{\varepsilon_k}}}(x)} (\nabla_\alpha u_{{{\varepsilon_k}}}(x),\tfrac{1}{{{\varepsilon_k}}} \nabla_3 u_{{{\varepsilon_k}}}(x)) dx \geq  \int_{\omega}a_0(x_\alpha)Q(W_0^{p_0(x_\alpha)})(\nabla_{\alpha}  u(x_\alpha))dx_{\alpha}, 
\end{align*}
for every $u \in W^{1,p_0(\cdot)}_0(\omega;\mathbb R^3)$, which in view of the definition of $\mathcal J$ in \eqref{Jqca(x)} concludes the proof.
%\nabla v_k\wto \nabla v$ weakly in $L^1(\Omega, \R^d)$.
\qed

\color{black}

\begin{prop} \label{ubqcxdimredp}
Under the assumptions of Theorem {\rm \ref{ubqcxdimred}}, 
 %for every $ u \in W_0^{1, p_0(\cdot)}(\omega;\mathbb R^3)$  there exists  a sequence $\{\tilde u_{{\eps_k}} \} \subseteq W_0^{1{,\color{blue}p_0^-}}(\omega;\mathbb R^{3})$ satisfying   $\tilde u_{{\eps_k}}\to u$ in $L^{p_0(\cdot)}(\omega;\mathbb R^3)$ and such that  $$\mathcal I(u)\geq \limsup_{k\to \infty} \mathcal I_{\varepsilon_k}(\tilde{u}_{\eps_{k}})$$
%where $\mathcal I_\varepsilon$ and $\mathcal I$ are defined by \eqref{Jepsf} and \eqref{Jqc}, respectively.
%\color{blue} Riscriverei così:  Under the same assumptions of Theorem \ref{quasiconvex} 
it results that
$$\Gamma(L^1)\hbox{-}\limsup_{k \to + \infty} \mathcal J_{{{\varepsilon_k}}}\leq \mathcal J,$$
where  $\{\mathcal J_{\varepsilon_k}\}$ and $\mathcal J$ are the functionals defined by \eqref{Jepsfa(x)} and \eqref{Jqca(x)}, respectively. 
	
\end{prop}
\color{black}
\begin{proof}
%	We aim at proving \eqref{Gammalimsup}
	%
	%\Gamma(L^{p_0(x_\alpha)})-\limsup_{\varepsilon} {\mathcal J}_\varepsilon(u) \le \, \limsup_{\varepsilon} {\mathcal J}_{\varepsilon}(u_{\varepsilon}) = {\mathcal J}(u)
	%\]
	%under the extra assumption that $f_0$ in \eqref{f0ours} coincides with 
	%\begin{equation}\label{f0ours2}f_0(x_\alpha,y, p, \xi_\alpha):= \inf_{\xi_3 \in \mathbb R^3}f(x_\alpha,y, p, \xi_\alpha,\xi_3) 
	%\end{equation} 
  Without loss of generality, we can restrict to consider the case when $\mathcal J(u)<+\infty$.
	
	 Taking into account \eqref{a0def1}, up to some minor modifications,  we can reason as in the first part of Proposition \ref{Gammlisup} and  obtain that
 
	\begin{equation}\label{above2}
	I(u) \le \, \inf_{w \in L^{p_0(\cdot)} (\omega; \mathbb{R}^3)} \int_{\omega} a_0(x_{\alpha}) W^{ p_0(x_{\alpha})}(\nabla u(x_\alpha), w(x_\alpha)) \, d x_{\alpha}
 \qquad \forall u \in \mathcal{C}^{\infty}_0(\omega; \mathbb{R}^3)\end{equation}
	where 
 %\begin{align}\label{JGammal}
 \[
 I:=\Gamma(L^{p_0(\cdot)}) \hbox{-} {\limsup_{k \to + \infty} }\mathcal{J}_{{{\varepsilon_k}}}.
 \]
 Following the same argument in \cite[page 558]{LDR}, 
there exists a measurable function $w_0=w_0(x_{\alpha})$  such that 
\[
W_0( \nabla u(x_{\alpha})) = \inf_{\xi_3 \in \mathbb R^3} [W( \nabla u(x_{\alpha}), \xi_3)]=W(\nabla u(x_{\alpha}), w_0(x_{\alpha})) \qquad \mbox{ for } \mathcal L^2-\hbox{ a.e. } x_{\alpha}\in \omega.
\]
In particular, by using \eqref{W0p}, we deduce that
\begin{align*}
	a_0(x_{\alpha}) W_0^{p_0(x_{\alpha}) }(\nabla_\alpha u(x_{\alpha})) = a_0(x_{\alpha}) W^{p_0(x_{\alpha})}(\nabla_\alpha u(x_{\alpha}), w_0(x_{\alpha})) =
 \inf_{\xi_3 \in \mathbb R^3} [a_0(x_{\alpha}) W^{p_0(x_{\alpha})}(\nabla_\alpha u(x_{\alpha}), \xi_3)].
	\end{align*}
Moreover, thanks to \eqref{coerca(x)},  such a  function $w_0$ belongs to $L^{p_0(\cdot)}(\omega;\mathbb R^3)$. 
Hence, by using \eqref{above2}, we get that 
%\begin{equation}\label{eq:dentroC2}
\[
I(u)  \leq \mathcal J_1(u) \qquad \forall u \in C^{\infty}_0(\omega; \mathbb{R}^3)
\]
%\end{equation}
where $\mathcal J_1:L^{1}(\omega;\mathbb R^3)\to [0, +\infty]$ is the functional defined by 

	\[ 
	{\mathcal J}_1(u) := \left\{
 	\begin{array}{ll}\displaystyle  \int_{\omega} a_0(x_{\alpha})  W^{p_0(x_{\alpha})}_0(\nabla_{\alpha} u(x_{\alpha})) \, d x_{\alpha} &\hbox{ if } u\in W^{1,p_0(\cdot)}_0(\omega;\mathbb R^3) \\
	\\
 		+\infty &\hbox{ otherwise in }L^{1}(\omega;\mathbb R^3).
 	\end{array}
 	\right.
\]

 Since $a_0\in L^\infty(\omega)$ and thanks to the growth condition 
 $$
 W_0(\xi_\alpha) \leq C_2(|\xi_\alpha|+1) \hbox{ for every }\xi_\alpha \in \mathbb R^{3 \times 2},
 $$ 
 inherited by \eqref{Wlingrowth},  ${\mathcal J}_1$ results strongly continuous  on  $W^{1, p_0(\cdot)}_0(\omega;\mathbb R^3)$.
\\
%&= \, \int_{\omega} a_0(x_{\alpha})  W_0^{p_0(x_{\a})} (\nabla_\alpha u(x_{\alpha})) \, d x_{\alpha}
	%\end{align*}
	%\textcolor{magenta}{questo ultimo passaggio non dovrebbe funzionare a meno che la $f$ non dipende da $x_3$ oppure a meno che non ci siano delle ipotesi che ci garantiscono che il minimo non dipende da $x_3$} A me sembra funzionare 
 By using  the density of $C_0^{\infty}(\omega;\mathbb R^3)$ in $W^{1, p_0(\cdot)}_0(\omega;\mathbb R^{3}) $ (see Part (1) in Remark \ref{remdens}) and the lower semicontinuity of $I$   %{\color{blue} strong $L^{1}(\Omega;\mathbb R^3) $ 
		with respect to the strong convergence in $W^{1, p_0(\cdot)}_0(\omega;\mathbb R^{3}) $  we  obtain that 
  \beq\label{primadis}
		I(u) \le \,\mathcal J_1(u)  \quad \forall u \in L^{1}(\omega;\mathbb R^3).
  \eeq
%Now we can introduce  $$J(u):=\left\{\begin{array}{ll} \displaystyle \int_\omega a_0(x_{\alpha}){W}^{p_0(x_{\alpha})}_0(x_\alpha, \nabla_\alpha u(x_\alpha))dx_\alpha , &\hbox{ if }{W}^{p_0(x_{\alpha})}_0(x_\alpha, \nabla_\alpha u(x_\alpha))dx_\alpha ,\in L^1(\omega) \\
%\\+\infty &\hbox{ otherwise in }L^{1}(\omega;\mathbb R^3).\end{array}\right.
%$$
Now, by  reasoning as in the proof of Proposition \eqref{gammalimsupquasic}, we can show that that  there exists a  sequence $\tilde u_{{\varepsilon_k}} \in W^{1,p_0(\cdot)}_{0}(\omega;\mathbb R^3)$  such that    $\tilde u_{{\varepsilon_k}}\wto u$ in $L^{ p_0(\cdot)}(\omega;\mathbb R^3)$ 
 and \begin{equation}\label{secdis}
 \lim_{k \to + \infty} \mathcal J_1(\tilde u_{{\varepsilon_k}}) =\lim_{k \to +\infty} \int_\omega a_0(x_{\alpha})W_0^{p_0(x_\alpha)}(x_\alpha ,\nabla \tilde u_{{\varepsilon_k}}(x_\alpha)) d x_\alpha=  \int_\omega Q (a_0(x_{\alpha}) W_0^{p_0(x_\alpha)})(x_\alpha, \nabla u(x_\alpha))d x_\alpha.
 \end{equation}
Since $I$ is lower semicontinuous with respect to the $L^{p_0(\cdot)}$-convergence and  $Q\left( a_0(x_\alpha)W_0^{p_0(x_\alpha)}(x_{\alpha},\cdot)\right)= a_0(x_\alpha)Q\left(W_0^{p_0(x_\alpha)}(x_{\alpha},\cdot)\right),$ for $\mathcal L^2$-a.e. $x_\alpha \in \omega$,  \eqref{secdis} combined with \eqref{primadis} implies

$$I(u)\leq \liminf_{k \to + \infty} I(\tilde u_{{\varepsilon_k}}) \leq  \lim_{k \to +\infty} \mathcal J_1(\tilde u_{{\varepsilon_k}}) = \lim_{k \to + \infty}  J(\tilde u_{{\varepsilon_k}}) =\int_\omega Q (W_0^{p_0(x_\alpha)})(x_\alpha, \nabla \tilde u(x_\alpha))d x_\alpha= \mathcal I(u).
$$  
and the proof is over.

 \end{proof}
 
% \textcolor{magenta}{questo potrebbe essere un corollario come il caso di prima con la $W$ continua nella variabile spaziale tutta $(x_{\alpha}, x_3)$ però la lower bound non la sappiamo fare...}

\begin{proof}[Proof of Theorem \textnormal{\ref{ubqcxdimred}}]
It follows from Propositions \textnormal{\ref{semicon2dimred}} and \textnormal{\ref{ubqcxdimredp}}.
    
\end{proof}

	%\color{red} CONTROLLARE LE NOTAZIONI DEI FUNZIONALI SOTTO anche se li ho modificati.
 
Analogously to Corollary \ref{dimredconvcor}, the following results hold. The proof is omitted, being very similar to the one presented in the convex case. 
 \begin{cor}
	%\label{dimredconvcornotcx}
 Let $\mathcal I_{\varepsilon_k} , \mathcal I, \mathcal J_{\varepsilon_k}$ and $\mathcal J$ be the functionals in  \eqref{Jepsf}, \eqref{Jqc},\eqref{Jepsfa(x)}, and \eqref{Jqca(x)}, respectively.
 Under the same assumptions and with the  same notation of Theorems {\rm \ref{quasiconvex}} and {\rm \ref{ubqcxdimred}} it results that
	$$\Gamma(L^{p_0(\cdot)})\hbox{-}\lim_{k \to + \infty}\mathcal I_{\varepsilon_k}= \mathcal I,
	$$
 and $$\Gamma(L^{p_0(\cdot)})\hbox{-}\lim_{k \to + \infty} \mathcal J_{\varepsilon_k}= \mathcal J.$$
  %color{OliveGreen}
	
\end{cor}

\color{black}

\noindent{\bf Acknowledgments.}  
The authors are members of GNAMPA-INdAM, whose support is gratefully acknowledged, through the projects 
 ``Prospettive nella scienza dei materiali: modelli variazionali, analisi
asintotica e omogeneizzazione'' (2023),  ``Analisi variazionale di modelli non-locali
nelle scienze applicate'' (2020) and ``Problemi ellittici e parabolici con singolarità: esistenza, regolarità ed omogeneizzazione'' (2020). The work of ME has been partially supported by PRIN 2020 ``Mathematics for industry 4.0 (Math4I4)'' .
FP thanks Dipartimento di Scienze di Base ed Applicate per l'Ingegneria at Sapienza - Universit\`a di Roma for its hospitality.
EZ is grateful to Dipartimento di Scienze Agrarie, Alimentari e Agro-ambientali
at Universit\`a di Pisa, whose hospitality is gratefully acknowledged.
FP and EZ are indebted with Dipartimento di Scienze Fisiche Informatiche e Matematiche at University of Modena and Reggio Emilia, whose support is kindly acknowledged.
The work of EZ is also supported by Sapienza - University of Rome through the projects Progetti di ricerca medi, (2021), coordinator  S. Carillo e Progetti di ricerca piccoli,  (2022), coordinator E. Zappale.

\end{document}